\documentclass[12pt,reqno]{amsart}
\usepackage{amscd,amssymb,amsmath,color}
  \usepackage[all]{xy}
\numberwithin{equation}{section}
\usepackage{amsfonts}
\usepackage{amsthm}
\usepackage{color}
\usepackage{float}
\usepackage{hyperref}
\usepackage{todonotes}
\usepackage{blkarray}
\usepackage{pgf,tikz}
\usepackage{mathrsfs}
\usetikzlibrary{arrows}
\definecolor{cqcqcq}{rgb}{0.7529411764705882,0.7529411764705882,0.7529411764705882}
\definecolor{qqzzff}{rgb}{0.,0.6,1.}
\definecolor{wwccqq}{rgb}{0.4,0.8,0.}
\definecolor{ffqqqq}{rgb}{1.,0.,0.}

\input{prepictex.tex}
\input{pictex.tex}
\input{postpictex.tex}

\usepackage{todonotes}
\usepackage{pgf,tikz}

\title{A Quantum Twistor Bundle}

\author{Sophie Emma Zegers}
\address{Charles University, Sokolovská 49/83, 186 75 Praha 8, Czech Republic} 
\email{sophieemmazegers@gmail.com}

\thanks{This work was  supported by  the 
DFF-Research Project 2 on `Automorphisms and invariants of operator algebras', Nr. 7014--00145B and by the Carlsberg Foundation through an Internationalisation Fellowship.}

\date\today



\begin{document}

\begin{abstract}
We investigate a quantum twistor bundle constructed as a $U(1)$-quotient of the quantum instanton bundle of Bonechi, Ciccoli and Tarlini. 
It is an example of a noncommutative bundle fulfilling conditions of the purely algebraic framework proposed by Brzezi\'{n}ski and 
Szyma\'{n}ski. We provide a detailed description of the corresponding $C^*$-algebra of `continuous functions' on its noncommutative total space. 
\end{abstract}

\theoremstyle{plain}
\newtheorem{theorem}{Theorem}[section]
\newtheorem{corollary}[theorem]{Corollary}
\newtheorem{lemma}[theorem]{Lemma}
\newtheorem{proposition}[theorem]{Proposition}
\newtheorem{conjecture}[theorem]{Conjecture}
\newtheorem{commento}[theorem]{Comment}
\newtheorem{problem}[theorem]{Problem}
\newtheorem{remarks}[theorem]{Remarks}

\theoremstyle{definition}
\newtheorem{example}[theorem]{Example}

\newtheorem{definition}[theorem]{Definition}

\newtheorem{remark}[theorem]{Remark}

\newcommand{\Nb}{{\mathbb{N}}}
\newcommand{\Rb}{{\mathbb{R}}}
\newcommand{\Tb}{{\mathbb{T}}}
\newcommand{\Zb}{{\mathbb{Z}}}
\newcommand{\Cb}{{\mathbb{C}}}

\newcommand{\CP}{\mathbb{C}P}

\newcommand{\Af}{\mathfrak A}
\newcommand{\Bf}{\mathfrak B}
\newcommand{\Ef}{\mathfrak E}
\newcommand{\Gf}{\mathfrak G}
\newcommand{\Hf}{\mathfrak H}
\newcommand{\Kf}{\mathfrak K}
\newcommand{\Lf}{\mathfrak L}
\newcommand{\Mf}{\mathfrak M}
\newcommand{\Rf}{\mathfrak R}

\newcommand{\x}{\mathfrak x}
\def\C{\mathbb C}
\def\N{\mathbb N}
\def\R{\mathbb R}
\def\T{\mathbb T}
\def\Z{\mathbb Z}

\def\A{{\mathcal A}}
\def\B{{\mathcal B}}
\def\D{{\mathcal D}}
\def\E{{\mathcal E}}
\def\F{{\mathcal F}}
\def\G{{\mathcal G}}
\def\H{{\mathcal H}}
\def\J{{\mathcal J}}
\def\K{{\mathcal K}}
\def\LL{{\mathcal L}}
\def\N{{\mathcal N}}
\def\M{{\mathcal M}}
\def\N{{\mathcal N}}
\def\OO{{\mathcal O}}
\def\P{{\mathcal P}}
\def\Q{{\mathcal Q}}
\def\SS{{\mathcal S}}
\def\T{{\mathcal T}}
\def\U{{\mathcal U}}
\def\W{{\mathcal W}}

\def\ext{\operatorname{Ext}}
\def\span{\operatorname{span}}
\def\clsp{\overline{\operatorname{span}}}
\def\Ad{\operatorname{Ad}}
\def\ad{\operatorname{Ad}}
\def\tr{\operatorname{tr}}
\def\id{\operatorname{id}}
\def\en{\operatorname{End}}
\def\aut{\operatorname{Aut}}
\def\out{\operatorname{Out}}
\def\coker{\operatorname{coker}}
\def\pol{\mathcal{O}}%{\operatorname{Pol}}

\def\la{\langle}
\def\ra{\rangle}
\def\rh{\rightharpoonup}
\def\cl{\textcolor{blue}{$\clubsuit$}}

\def\bst{\textcolor{blue}{$\bigstar$}}

%Sophie new command
\newcommand{\Aut}[1]{\text{Aut}(#1)}
\newcommand{\Fn}{\mathcal{F}_n}
\newcommand{\On}{\mathcal{O}_n}
\newcommand{\Dn}{\mathcal{D}_n}
\newcommand{\Sn}{\mathcal{S}_n}

\newcommand{\norm}[1]{\left\lVert#1\right\rVert}
\newcommand{\inpro}[2]{\left\langle#1,#2\right\rangle}

\renewcommand{\bibname}{References}

\maketitle

\addtocounter{section}{-1}

%%%%%%%%%%%%%%%%%%%%%%%%%%%%%%%%%%%%%%%%%%%%%%%%%%%%%%%%%%%%%
%%%%%%%%%%%%%%%%%%%%%%%%%%%%%%%%%%%%%%%%%%%%%%%%%%%%%%%%%%%%%

\section{Introduction}

%%%%%%%%%%%%%%%%%%%%%%%%%%%%%%%%%%%%%%%%%%%%%%%%%%%%%%%%%%%%%
%%%%%%%%%%%%%%%%%%%%%%%%%%%%%%%%%%%%%%%%%%%%%%%%%%%%%%%%%%%%%

Locally trivial bundles play a crucial role in classical Geometry and Topology, and not surprisingly a great deal of effort has been extended by 
numerous researchers with the aim of generalizing these classical concepts to the realm of Noncommutative Geometry. Those efforts have been 
very successful with regard to compact principal  bundles and associated vector bundles, equipped with a quantum structure group. See for example 
\cite{schneider},  \cite{brzmaj} and \cite{brzhaj} for development of the purely algebraic theory, and \cite{decyam} for an excellent treatment of 
freeness of actions of compact quantum groups on $C^*$-algebras. 
In the noncommutative case, a principal bundle at the $C^*$-algebraic level consists of a unital $C^*$-algebra $A$ which plays the role of the total space,  
equipped with a free action of a compact quantum structure group $G$. The corresponding fixed point algebra $A^G$ plays the role of the base space. 
However, the general concept of a noncommutative fibre bundle is far from understood. An obvious obstruction in imitating the classical approach 
is lack of transition functions subordinated to a local trivialisation, and thus lack of a clear path from the fibration to a group. 

Recently, a general framework for understanding a large class of noncommutative, locally trivial fibre bundles has been proposed in 
\cite{BrzSz:ncb}. The motivation for that work comes from the following classical case. Given a compact principal bundle $G\to M\to B$ and 
a closed subgroup $H\subseteq G$, one may pass to a bundle $G/H\to M/H\to B$. In \cite{BrzSz:ncb}, the discussion is carried out 
primarily on the purely algebraic level and at the high level of generality allowing for $G$ to be replaced by a coalgebra. However, it is important 
to keep $C^*$-algebraic perspective as well and indeed, the most interesting examples described so far involve both purely algebraic as well 
as analytic components based on $C^*$-algebras. Here we should mention the quantum flag manifold $SU_q(3)/{\mathbb T}^2$ with fibres 
standard Podle\'{s} spheres studied in \cite{BrzSz:BiaProc}  and \cite{BrzSz:flag}, as well as the same quantum manifold but with  
fibres generic Podle\'{s} spheres given in \cite{BrzSz:ncb}. 

Yet another example briefly described in \cite{BrzSz:ncb} on the purely algebraic level is a quantum twistor bundle $\C P_q^1 \to \C P_q^3 \to S^4_q$, 
derived from the principal quantum instanton bundle $SU_q(2) \to S_q^7 \to S^4_q$ constructed by Bonechi, Ciccoli, D\c{a}browski and Tarlini, 
\cite{bcdt} and \cite{bct1}. 
It should be noted that the action of $SU_q(2)$ in this case is not an algebra homomorphism, which makes the construction 
difficult and interesting at the same time.  The main objective of the present paper is to provide a thorough analysis of $C^*$-algebraic aspects of the quantum twistor bundle. We should point out that the $C^*$-algebra of the total space of our twistor 
bundle is given by generators satisfying quite different 
relations from those of the Vaksman-Soibelman complex projective space. Nevertheless, we show that on the level of $C^*$-algebras these two 
are isomorphic. 

The structure of the paper is as follows: In Section \ref{preliminaries} we recall definitions on graph $C^*$-algebras and freeness of actions on $C^*$-algebras. We then recall, in Section \ref{QuantumInstantonBundle}, the structure of the quantum instanton bundle due to Bonechi, Ciccoli,  D\c{a}browski and Tarlini. Section \ref{TwistorBundle} contains a detailed analysis of the enveloping $C^*$-algebra  of the total space for the quantum twistor bundle from \cite{BrzSz:ncb}. Moreover, we end the section by showing that the circle action defining the total space of the quantum twistor bundle is free. Finally, in Section \ref{projections} we take a look at the $K$-theory of the total space of the bundle and provide a natural construction of polynomial projectors. 
\vspace{0.2cm}

\textbf{Acknowledgment}. We are grateful to Francesco D'Andrea and Giovanni Landi for useful comments on an earlier version of this paper. 

%%%%%%%%%%%%%%%%%%%%%%%%%%%%%%%%%%%%%%%%%%%%%%%%%%%%%%%%%%%%%
%%%%%%%%%%%%%%%%%%%%%%%%%%%%%%%%%%%%%%%%%%%%%%%%%%%%%%%%%%%%%

\section{Preliminaries}\label{preliminaries}
\subsection*{Graph $C^*$-algebras} A directed graph $E=(E^0,E^1,r,s)$ consists of a countable set $E^0$ of \textit{vertices}, a countable set $E^1$ of \textit{edges} and two maps $r,s: E^1\to E^0$ called the \textit{range map} and the \textit{source map} respectively. For an edge $e\in E^1$ from $v$ to $w$ we have $s(e)=v$ and $r(e)=w$. 

Let $E=(E^0,E^1,r,s)$ be a directed graph. The graph $C^*$-algebra $C^*(E)$ is defined (see e.g. \cite{bpis,flr00}) as the universal $C^*$-algebra generated by families of orthogonal projections $\{P_v | \ v\in E^0\}$ and partial isometries $\{S_e | \ e\in E^1\}$ with mutually orthogonal ranges (i.e. $S_e^*S_f=0, e\neq f$) 
subject to the Cuntz-Krieger relations 
\begin{itemize}
\item[] (CK1) $S_e^*S_e=P_{r(e)}$
\item[] (CK2) $S_eS_e^*\leq P_{s(e)}$
\item[] (CK3) $P_v=\underset{s(e)=v}{\sum}S_e S_e^*$, if $\{e\in E^1| \ s(e)=v\}$ is finite and nonempty. 
\end{itemize}
A \textit{path} $\alpha$ in a graph is a finite sequence $\alpha=e_1e_2\cdots e_n$ of edges satisfying $r(e_i)=s(e_{i+1})$ for $i=1,...,n-1$ and we let $S_\alpha=S_{e_1}S_{e_2}\cdots S_{e_n}$.

A graph $C^*$-algebra $C^*(E)$ admits by universality a circle action, called the \textit{gauge action}, $\gamma: U(1)\to \text{Aut}(C^*(E))$ for which 
$\gamma_z(P_v)=P_v \ \text{and} \ \gamma_z(S_e)=zS_e$
for all $v\in E^0, e\in E^1$ and $z\in U(1)$.

In \cite{vs} Vaksman and Soibelman defined a quantum version of the odd sphere denoted $S_q^{2n+1}$ with $q\in (0,1)$ (see \eqref{VS} for the definition). We refer to the $C^*$-algebra $C(S_q^{2n+1})$ as the \textit{space of continuous functions} on the virtual space $S_q^{2n+1}$. Hence, the notation of continuous functions must be understood in an abstract sense where there is no underlying space. In \cite{hs} it is shown for $q\in [0,1)$ that $C(S_q^{2n+1})$ is isomorphic to $C^*(L_{2n+1})$ where the graph $L_{2n+1}$ has vertices $v_i, i=0,...,n$ and edges $e_{ij}, 0\leq i\leq j\leq n$ such that $s(e_{ij})=v_i$ and  $r(e_{ij})=v_j$. In particular this result shows that $C(S_q^{2n+1})$ are all isomorphic for $q\in [0,1)$.  

The complex projective space $C(\CP_q^n)$ is defined as the fixed point algebra under the circle action on $C(S^{2n+1}_q)$  given on generators by $z_i\mapsto wz_i, w\in U(1)$. Under the isomorphism between $C(S_q^{2n+1})$ and $C^*(L_{2n+1})$, the $U(1)$-action on $C(S_q^{2n+1})$ becomes the gauge action, from which $C(\CP_q^n)$ can be described as a graph $C^*$-algebra (see \cite[Section 4.3]{hs}). 

We will in the present paper consider a $C^*$-algebra, denoted $C(\CP_q^{3,\mu})$, defined as the fixed point algebra of another circle action on the quantum 7-sphere. Even though the action is different from the gauge action, we will in Section \ref{TheCalgebra} prove that $C(\CP_q^{3,\mu})\cong C(\CP_q^3)$. 

\subsection*{Free actions on a $C^*$-algebra}
The notion of a free action of a compact group on a topological space has been translated into the noncommutative setting by considering compact quantum groups instead of compact groups. The description has been formulated in two different ways. The first is due to Rieffel \cite{R90} and the second is given by Ellwood  \cite{E00}.  In \cite{decyam} it was shown that these two notions of freeness are equivalent. We also refer to the paper \cite{BDH17} for a good treatment of freenes. 

%An action (coaction) of a compact quantum group $C(G)$ (in particular continuous functions on a classical compact group) on a $C^*$-algebra $A$ is defined as a unital injective $*$-homomorphism 
%$$
%\rho: A\to A\otimes C(G)
%$$
%such that 
%\begin{itemize}
%\item[(i)] $(\rho\otimes id)\circ\rho=(id\otimes \Delta)\circ\rho$ as maps from $A$ to $A\otimes C(G)\otimes C(G)$.
%\item[(ii)] $\{\rho(a)(1\otimes x)| a\in A, x\in C(G)\}$ is dense in $A\otimes C(G)$. 
%\end{itemize}
%If $G$ is a compact group acting on a compact Hausdorrf space $X$, then by letting $A=C(X)$, condition (i) corresponds to the compatibility condition and (ii) to the identity condition of the action of $G$ on $X$.

A free action of a compact quantum group on a $C^*$-algebra is defined as follows.
\begin{definition}[\cite{E00}]\label{Free}
The coaction $\rho: A\to A\otimes C(G)$ of a compact quantum group on a $C^*$-algebra is \textit{free} if the image of the linear map
$$
\Phi: A\otimes A\to A\otimes C(G), \ a_1\otimes a_2\mapsto (a_1\otimes 1)\rho(a_2)
$$
is dense in $A\otimes C(G)$. 
\end{definition} 
By letting $A=C(X)$ and dualising the action, the definition is in the classical case equivalent to an action of a compact group $G$ on a topological space $X$ being free. 

The action of a compact group on a $C^*$-algebra is free if the corresponding coaction is free in the sense of Definition \ref{Free}.  E.g. Let $u$ be the canonical generator of $C(S^1)$ then the coaction $\hat{\gamma}: C^*(E)\to C^*(E)\otimes C(S^1) $ corresponding to the gauge action is given by 
$$
\hat{\gamma}(S_e)= S_e\otimes u, \ \ \hat{\gamma}(P_v)=P_v\otimes 1. 
$$
In \cite[Proposition 2]{S01} a condition on the graph $E$, which guarantees that the gauge action on $C^*(E)$ is free, is given. Since the graph $L_{2n+1}$ satisfies the condition, we obtain that the circle action on the quantum sphere by Vaksman and Soibelman, defining the quantum complex projective space, is free. We will in Section \ref{Freenes} prove that the new circle action on $C(S_q^7)$, defining the main object of interest $C(\CP_q^{3,\mu})$, is also a free action.

\section{The quantum instanton bundle} \label{QuantumInstantonBundle}

%%%%%%%%%%%%%%%%%%%%%%%%%%%%%%%%%%%%%%%%%%%%%%%%%%%%%%%%%%%%%
%%%%%%%%%%%%%%%%%%%%%%%%%%%%%%%%%%%%%%%%%%%%%%%%%%%%%%%%%%%%%

To begin with, we briefly recall the structure of a quantum analogue of the instanton bundle $SU(2) \to S^7 \to S^4$ due to Bonechi, Ciccoli,  D\c{a}browski and Tarlini,
as given in \cite{bct1} and \cite{bcdt}. In Section \ref{TwistorBundle} we will construct a quantum version of the twistor bundle from this particular quantum instanton bundle.  We remark that in the litterature there exists other constructions of quantum instanton bundles, for example the one described by Pflaum in \cite{P94} and the one by Landi, Pagani and Reina in \cite{lpr}. The bundle from \cite{lpr} has the advantage over the one from \cite{bct1} 
insofar as the action of $SU_q(2)$ is by 
an honest $*$-homomorphism. Since the enveloping $C^*$-algebra of the total space of this bundle is not completely understood,  we will not use this bundle to define our quantum twistor bundle. 

Let $\mathcal{O}(U_q(4))$ be the polynomial algebra of the quantum unitary group, with $q\in(0,1)$. 
This is a universal algebra generated by $\{t_{ij}\}^4_{i,j=1}$ and $D_q^{-1}$, subject to the following relations: 
$$
\begin{aligned}
&t_{ik}t_{jk}=qt_{jk}t_{ik}, \ \  t_{ki}t_{kj}=qt_{kj}t_{ki}, \ \  i<j \\ 
&t_{il}t_{jk}=t_{jk}t_{il},\ \  i<j, k<l, \\
&t_{ik}t_{jl}-t_{jl}t_{ik}=(q-q^{-1})t_{jk}t_{il}, \ \ i<j, k<l, \\ 
&D_qD_q^{-1}=D_q^{-1}D_q=1. 
\end{aligned}
$$
Here $D_q$ is the quantum determinant, defined as 
$$
D_q=\sum_{\sigma\in S_4}(-q)^{I(\sigma)} t_{\sigma(1)1}\cdots t_{\sigma(4)4}, 
$$
where $I(\sigma)$ denotes the number of inversed pairs and $S_4$ is the group of permutations of 4 symbols. This algebra equipped with the usual 
comultiplication $\Delta_{U_q(4)}$, counit $\varepsilon$ and antipode $S_{U_q(4)}$, see \cite[p. 311--314]{ks}, is a Hopf algebra: 
$$ \begin{aligned}
\Delta_{U_q(4)}(t_{ij}) & = \sum_{k}t_{ik}\otimes t_{kj},  \\
\varepsilon(t_{ij}) & = \delta_{ij}, \\
S_{U_q(4)}(t_{ij}) & = (-q)^{i-j}\sum_{\sigma\in S_3} (-q)^{I(\sigma)} t_{j_{\sigma(1)}i_1} t_{j_{\sigma(2)}i_2} t_{j_{\sigma(3)}i_3}. 
\end{aligned} $$
Here $\{j_1,j_2,j_3\}=\{1,\ldots,j-1,j+1,\ldots,4\}$ and $\{i_1,i_2,i_3\}=\{1,\ldots,i-1,i+1,\ldots,4\}$. 
Define the adjoint operation on $\pol(U_q(4))$ by 
$$ 
t_{ij}^* = S_{U_q(4)}(t_{ji}) \;\;\; \text{and} \;\;\; D_q^*=D_q^{-1}. 
$$
Then $\pol(U_q(4))$ is a Hopf $*$-algebra and its enveloping $C^*$-algebra , $C(U_q(4))$, is the $C^*$-algebra of continuous functions on the compact quantum group $U_q(4)$. Note that continuous functions on $U_q(4)$  must be understood in an abstract sense since there is no underlying space.

Now let $\mathcal{O}(S_q^7)$ be the $*$-subalgebra of $\pol(U_q(4))$ generated by $\{z_i=t_{4i} \mid i=1,\ldots,4\}$. Its enveloping 
$C^*$-algebra, $C(S_q^7)$, is the universal $C^*$-algebra with the following relations: 
\begin{equation}\label{VS}
\begin{aligned}
&z_iz_j=qz_jz_i, \;\;\; \text{for} \; i<j, \\ 
& z_j^*z_i=qz_iz_j^*, \;\;\; \text{for} \; i\neq j, \\
&z_k^*z_k=z_kz_k^*+(1-q^2)\sum_{j<k} z_jz_j^*, \\ 
& \sum_{k=1}^4 z_kz_k^*=1.  
\end{aligned} 
\end{equation}
Thus, $C(S_q^7)$ is the $C^*$-algebra of 'continuous functions' on the Vaksman-Soibelman quantum $7$-sphere, \cite{vs}. Note that the generator $z_i$ from \cite{bct1} and \cite{bcdt}
corresponds to the generator $z_{5-i}$ from \cite{hs}.  

%%%%%%%%%%%%%%%%%%%%%%%%%%%%%%%%%%%%%%%%%%%%%%%%%%%%%%%%%%%%%%%%%%%%%%%%%%

A certain  coalgebra surjection $\pi_{SU_q(2)}:\pol(U_q(4))\to\pol(SU_q(2))$ is investigated in \cite{bct1}. 
This map is not an algebra homomorphism but only a right 
$\pol(U_q(4))$-module map in the sense that
$$ \pi_{SU_q(2)}(a) = \pi_{SU_q(2)}(b) \;\; \text{implies} \;\; \pi_{SU_q(2)}(ac) = \pi_{SU_q(2)}(bc) $$
for all $a,b,c\in\pol(U_q(4))$. One has that  
$$ \rho_{U_q(4)} = (\id \otimes \pi_{SU_q(2)})\Delta_{U_q(4)} $$
is a right $\mathcal{O}(SU_q(2))$-coaction on $\mathcal{O}(U_q(4))$. Its restriction yields a right $\mathcal{O}(SU_q(2))$-coaction 
on $\mathcal{O}(S_q^7)$, namely  
$$ \rho_{S_q^7} : \mathcal{O}(S_q^7) \to \mathcal{O}(S_q^7)\otimes \mathcal{O}(SU_q(2)). $$
Note that $\rho_{S_q^7}$ is not a $*$-homomorphism. 

%%%%%%%%%%%%%%%%%%%%%%%%%%%%%%%%%%%%%%%%%%%%%%%%%%%%%%%%%%%%%%%%%%%%%%

Now, let $\pol(S_q^4)$ be the vector space of coinvariants with respect to $\rho_{S_q^7}$, i.e. 
$$
\mathcal{O}(S_q^4)=\{u\in \mathcal{O}(S_q^7) \mid \rho_{S_q^7}(u)=u\otimes1\}.
$$
It is  shown in \cite{bct1} that $\mathcal{O}(S_q^4)$ is a $*$-algebra generated by $\{a,b,R\}$, where 
$$ a= z_1z_4^* - z_2z_3^*, \;\; b= z_1z_3 + q^{-1}z_2z_4, \;\; R=z_1z_1^* + z_2z_2^* $$
satisfies the following relations:
$$
\begin{aligned}
&Ra=q^{-2}aR, \ Rb=q^2bR, \ ab=q^3ba, \ ab^*=q^{-1}b^*a, \\
&aa^*+q^2bb^*=R(1-q^2R), \ aa^*=q^2a^*a+(1-q^2)R^2, \\
&b^*b=q^4bb^*+(1-q^2)R.  
\end{aligned}
$$
As shown in \cite{bct1}, the $C^*$-algebra $C(S_q^4)$ generated by $\pol(S_q^4)$ is isomorphic to the minimal unitisation of the compacts. 
On the other hand, it was shown in \cite{bcdt} that $\pol(S_q^4) \subseteq \pol(S_q^7)$ is a coalgebra Galois extension (see \cite{BH99} for more on Galois extensions). In this sense, 
a quantum instanton bundle 
$$ SU_q(2) \to S_q^7 \to S_q^4 $$
has been created as a coalgebra Galois extension. 

%%%%%%%%%%%%%%%%%%%%%%%%%%%%%%%%%%%%%%%%%%%%%%%%%%%%%%%%%%%%%
%%%%%%%%%%%%%%%%%%%%%%%%%%%%%%%%%%%%%%%%%%%%%%%%%%%%%%%%%%%%%

\section{The quantum twistor bundle}\label{TwistorBundle}
We take as our starting point the quantum instanton bundle described in Section \ref{QuantumInstantonBundle}. Classically the twistor bundle $\CP^1\to \CP^3\to S^4$ is obtained as the quotient of the instanton bundle $SU(2) \to S^7\to S^4$ by the subgroup $U(1)$ of $SU(2)$. Motivated by the classical setting, we  construct a quantum version of the twistor bundle, at the $C^*$-algebraic level, by passing to the fixed point algebra for the circle action $\mu$ on $C(S_q^7)$, defined on the generators by 
\begin{equation}\label{newaction}
\mu_w(z_j) = 
\begin{cases}
wz_j, & j=1,4 \\
\overline{w}z_j, & j=2,3
\end{cases}
\end{equation}
for all $w\in U(1)$. First we remark that the action is defined in this way such that $C(S_q^4)$ is invariant under the action. Also note that by taking the quotient of $C(S_q^7)$ with the ideal generated by $z_1$ and $z_2$ gives us the $C^*$-algebra generated by $z_3, z_4$, which is isomorphic to $C(SU_q(2))=C^*(\alpha,\gamma)$ of Woronowicz, \cite{w1}, by the isomorphism $z_4\mapsto\alpha^*, z_3\mapsto \gamma$. 
The action $\mu$ restricted to $C(SU_q(2))$ is then given on the generators by 
$$
\mu_w(\alpha)=\overline{w}\alpha, \ \mu_w(\gamma)=\overline{w}\gamma.
$$
Hence $C(SU_q(2))^\mu$, serving as the base space of the bundle, is the Podlés sphere $C(\CP_q^1)$, \cite{P87}, which is isomorphic to the minimal unitisation of the compacts. 

Let $C(\C P_q^{3,\mu})$ be the fixed-point algebra of $C(S_q^7)$ for the  action $\mu$, and let $\pol(\C P_q^{3,\mu})$ be the intersection of $C(\C P_q^{3,\mu})$ with $\pol(S_q^7)$. We have $C(S_q^4) \subseteq C(\C P_q^{3,\mu})$ and $\pol(S_q^4) \subseteq \pol(\C P_q^{3,\mu})$. In Section \ref{TheCalgebra} we will analyse the $C^*$-algebra $C(\C P_q^{3,\mu})$ in details. 

By passing to the polynomial dense $*$-subalgebras of the above $C^*$-algebras we indeed obtain the quantum twistor bundle as described at the algebraic level in \cite{BrzSz:ncb} (see \cite[Proposition 3.5.3]{M21}).

\subsection{The $C^*$-algebra of the total space}\label{TheCalgebra}

%%%%%%%%%%%%%%%%%%%%%%%%%%%%%%%%%%%%%%%%%%%%%%%%%%%%%%%%%%%%%
%%%%%%%%%%%%%%%%%%%%%%%%%%%%%%%%%%%%%%%%%%%%%%%%%%%%%%%%%%%%%
The enveloping $C^*$-algebra of the fibre, $C(S_q^4)$, and the base space, $C(\CP_q^1)$, of the quantum twistor bundle are both known to be isomorphic to the minimal unitisation of the compacts. In this section we show (see Theorem \ref{TheCalgebra}) that the enveloping $C^*$-algebra of the total space $C(\CP_q^{3,\mu})$ is a well known graph $C^*$-algebra. 

Let $J_k$ be the closed 2-sided ideal of $C(\C P_q^{3,\mu})$ generated by $z_1z_1^*,...,z_kz_k^*$ for $k=1,2,3$. We clearly have that 
$C(\C P_q^{3,\mu})/J_3$ is the space of complex numbers. The quotient $C(\C P_q^{3,\mu})/J_2$ can be obtained by taking the quotient of $C(S_q^7)$ with the ideal generated by $z_1$ and $z_2$. This gives us the $C^*$-algebra generated by $z_3, z_4$, which is isomorphic to $C(SU_q(2))$.  The quotient $C(\C P_q^{3,\mu})/J_2$ is then the fixed point algebra
of the circle action $\mu$ on $C(SU_q(2))$ which is the Podlés sphere. 

We return now to the investigation of $C(\C P_q^{3,\mu})/J_1$. Consider the $C^*$-algebra generated by $z_2,z_3,z_4$, which is isomorphic to $C(S_q^5)$ 
of \cite{hs} by the renaming of the generators $z_2\mapsto z_3$, $z_2\mapsto z_2$ and $z_4\mapsto z_1$. It is shown in \cite[Theorem 4.4]{hs} that  
$C(S_q^5)$ is the graph $C^*$-algebra $C^*(L_5)$, corresponding to the following graph:
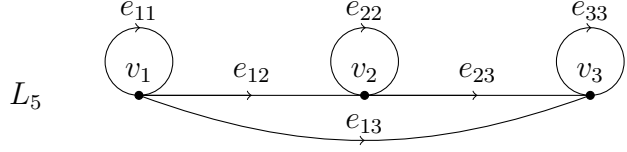
\begin{center}
\begin{tikzpicture}[scale=1.5]
\draw (0,0) -- (4,0);
\filldraw [black] (0,0) circle (1pt);
\filldraw [black] (2,0) circle (1pt);
\filldraw [black] (4,0) circle (1pt);
\draw [->] (0,0) -- (1,0);
\draw [->] (2,0) -- (3,0);

\draw[] (0,0) to [out=-20,in=-160] (4,0);
\draw [->] (1.99,-0.4) -- (2,-0.4);

\draw (0,0.3) circle [radius=0.3cm];
\draw (2,0.3) circle [radius=0.3cm];
\draw (4,0.3) circle [radius=0.3cm];
\draw [->] (-0.01,0.6) -- (0.01,0.6);
\draw [->] (1.99,0.6) -- (2.01,0.6);
\draw [->] (3.99,0.6) -- (4.01,0.6);

\node at (-1, 0)  {$L_5$};
\node at (0, 0.2)  {$v_1$};
\node at (2, 0.2)  {$v_2$};
\node at (4, 0.2)  {$v_3$};
\node at (1, 0.2)  {$e_{12}$};
\node at (3, 0.2)  {$e_{23}$};
\node at (2, -0.25)  {$e_{13}$};
\node at (0, 0.75)  {$e_{11}$};
\node at (2, 0.75)  {$e_{22}$};
\node at (4, 0.75)  {$e_{33}$};

\end{tikzpicture}
\end{center}

Under the isomorphism of $C(S_q^5)$ and $C^*(L_5)$, the action $\mu$ on $C^*(L_5)$ becomes 
$$
\mu_w(S_{e_{ij}})=\begin{cases}
wS_{e_{ij}} & i=1 \\
\overline{w}S_{e_{ij}} & i=2,3
\end{cases}, \ \ \ \mu_w(P_{v_i})=P_{v_i}, i=1,2,3
$$
for all $w\in U(1)$. Hence, $C(\C P_q^{3,\mu})/J_1$ is isomorphic to the fixed point algebra $C^*(L_5)^{\mu}$. 

\begin{theorem}
Let $F$ be the graph 
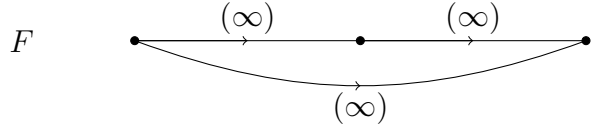
\begin{center}
\begin{tikzpicture}[scale=1.5]
\draw (0,0) -- (4,0);
\filldraw [black] (0,0) circle (1pt);
\filldraw [black] (2,0) circle (1pt);
\filldraw [black] (4,0) circle (1pt);
\draw [->] (0,0) -- (1,0);
\draw [->] (2,0) -- (3,0);

\draw[] (0,0) to [out=-20,in=-160] (4,0);
\draw [->] (1.99,-0.4) -- (2,-0.4);

\node at (-1, 0)  {$F$};
\node at (1, 0.2)  {$(\infty)$};
\node at (3, 0.2)  {$(\infty)$};
\node at (2, -0.6)  {$(\infty)$};
\end{tikzpicture}
\end{center}
with countably infinitely many edges from $w_i$ to $w_j$ whenever $i<j$. There exists a $C^*$-algebra isomorphism 
$$
\phi:C^*(F) \to C^*(L_5)^\mu  \cong C(\C P_q^{3,\mu})/J_1
$$
such that 
$$
\begin{aligned}
P_{w_j}&\mapsto P_{v_j}, \;\; j=1,2,3, \\
S_{\alpha_n}&\mapsto S_{e_{11}}^nS_{e_{12}}S_{e_{22}}^{n+1}, \;\; n\geq 0 \\
S_{\beta_{n,m}}&\mapsto S_{e_{11}}^nS_{e_{12}}S_{e_{22}}^mS_{e_{23}}S_{e_{33}}^{n-m}, \;\; n\geq m\geq 0,  \\
S_{\delta_n}&\mapsto S_{e_{11}}^nS_{e_{13}}S_{e_{33}}^{n+1}, \;\; n\geq 0 \\
S_{\gamma_n}&\mapsto S_{e_{22}}^nS_{e_{23}}{S_{e_{33}}^*}^{n+1}, \;\; n\geq 0,  
\end{aligned}
$$
where $\alpha_n, \beta_{n,m}, \delta_n, \gamma_n$ are the following edges: 
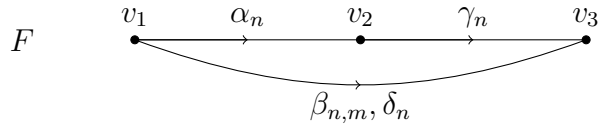
\begin{center}
\begin{tikzpicture}[scale=1.5]
\draw (0,0) -- (4,0);
\filldraw [black] (0,0) circle (1pt);
\filldraw [black] (2,0) circle (1pt);
\filldraw [black] (4,0) circle (1pt);
\node at (0,0.2) {$v_1$};
\node at (2,0.2) {$v_2$};
\node at (4,0.2) {$v_3$};
\draw [->] (0,0) -- (1,0);
\draw [->] (2,0) -- (3,0);

\draw[] (0,0) to [out=-20,in=-160] (4,0);
\draw [->] (1.99,-0.4) -- (2,-0.4);

\node at (-1, 0)  {$F$};
\node at (1, 0.2)  {$\alpha_n$};
\node at (3, 0.2)  {$\gamma_n$};
\node at (2, -0.6)  {$\beta_{n,m},\delta_n$};
\end{tikzpicture}
\end{center}
\end{theorem}
\begin{proof}
The target elements belong to the fixed point algebra and they satisfy the Cuntz-Krieger relations for the graph $F$. By universality of $C^*(F)$, there exists a $*$-homomorphism $\phi: C^*(F) \to C^*(L_5)$ as above. Injectivity of $\phi$ follows from the Cuntz-Krieger uniqueness theorem, \cite[Theorem 2]{flr00}, 
since the graph $F$ contains no loops and $\phi(P_{w_i})\neq 0$ for all $i$. It can be seen that $C^*(L_5)^{\mu}$ is generated by the following 
elements, with $n,m\geq 0$:
$$
\begin{aligned}
&S_{e_{11}}^n{S_{e_{11}}^n}^*, &&S_{e_{11}}^nS_{e_{12}}S_{e_{22}}^{n+1}, &&&S_{e_{22}}^nS_{e_{23}}S_{e_{33}}^m (S_{e_{22}}^mS_{e_{23}}S_{e_{33}}^n )^*,
&&&&S_{e_{11}}^nS_{e_{12}}S_{e_{22}}^mS_{e_{23}}S_{e_{33}}^{n-m}, m\leq n
 \\
&S_{e_{22}}^n{S_{e_{22}}^n}^*, &&S_{e_{11}}^nS_{e_{13}}S_{e_{33}}^{n+1}, &&&S_{e_{11}}^nS_{e_{12}}S_{e_{22}}^n(S_{e_{11}}^mS_{e_{12}}S_{e_{22}}^m)^*,
&&&&S_{e_{11}}^nS_{e_{13}}S_{e_{33}}^m(S_{e_{11}}^nS_{e_{13}}S_{e_{33}}^m)^*,
\\
&S_{e_{33}}^n{S_{e_{33}}^n}^*, &&S_{e_{22}}^nS_{e_{23}}{S_{e_{33}}^*}^{n+1}, &&&S_{e_{11}}^nS_{e_{12}}S_{e_{22}}^m(S_{e_{11}}^nS_{e_{12}}S_{e_{22}}^m)^*,
&&&&S_{e_{11}}^nS_{e_{13}}S_{e_{33}}^n(S_{e_{11}}^mS_{e_{13}}S_{e_{33}}^m)^*
\end{aligned}
$$

$$ S_{e_{11}}^kS_{e_{22}}^nS_{e_{23}}S_{e_{33}}^m(S_{e_{11}}^lS_{e_{22}}^sS_{e_{23}}S_{e_{33}}^t)^*,   \text{ when  }\, k+s+t=l+n+m.
$$
It can be shown that all these generators can be expressed by the images under $\phi$ of the generators for $C^*(F)$. Thus $\phi$ is surjective. 
\end{proof}
Note that $C(\C P_q^{3,\mu})/J_1$ is then isomorphic to $C(\CP_q^2)$ by \cite{hs}. 

\begin{theorem}\label{ideal}
The ideal $J_1$ in $C(\C P_q^{3,\mu})$ generated by $z_1z_1^*$ is isomorphic to the space of compact operators $\mathcal{K}$ on a separable Hilbert space. 
\end{theorem}
\begin{proof}
The strategy of the proof is to show that the ideal $J_1$ admits precisely one irreducible representation up to unitary equivalence. Then it follows by \cite{Na48, Ro53} that $J_1$ is isomorphic to the space of compact operators, since the only $C^*$-algebra which admits precisely one irreducible representation up to unitary equivalence is the space of compact operators. 

We assume that $q\in (0,1)$, the case $q>1$ can be done in a similar way.  
Let $\rho$ be an irreducible representation of $C(\C P_q^{3,\mu})$ such that $\rho|_{J_1}\neq 0$. Then  $\rho$ is also an irreducible representation of $J_1$,  \cite[Theorem 1.3.4]{a76}. We have 
\begin{equation}\label{z4}
z_4^*z_4=q^2z_4z_4^*+(1-q^2), 
\end{equation}
and hence $z_4z_4^*$ and $z_4^*z_4$ commute. 

Recall that the joint spectrum between two commuting elements $A$ and $B$ in a $C^*$-algebra, denoted $\sigma(A,B)$, is the collection of $(w(A),w(B))$ for all characters $w$ of the commutative $C^*$-algebra generated by $A$ and $B$.

By \eqref{z4}, the joint spectrum $\sigma(z_4z_4^*, z_4^*z_4)$ is contained in the straight line from the point $(1-q^2,0)$ to $(1,1)$. Hence the red line segment in Figure \ref{jointspectrum} is not in 
$\sigma(z_4^*z_4)$. Since $\sigma(z_4z_4^*) \cup\{0\} = \sigma(z_4^*z_4) \cup\{0\}$ \cite[Proposition 3.2.8]{kr1}, the green line segment is not in $\sigma(z_4z_4^*)$. 
Then the blue line segment cannot be in $\sigma(z_4^*z_4)$ and so on.

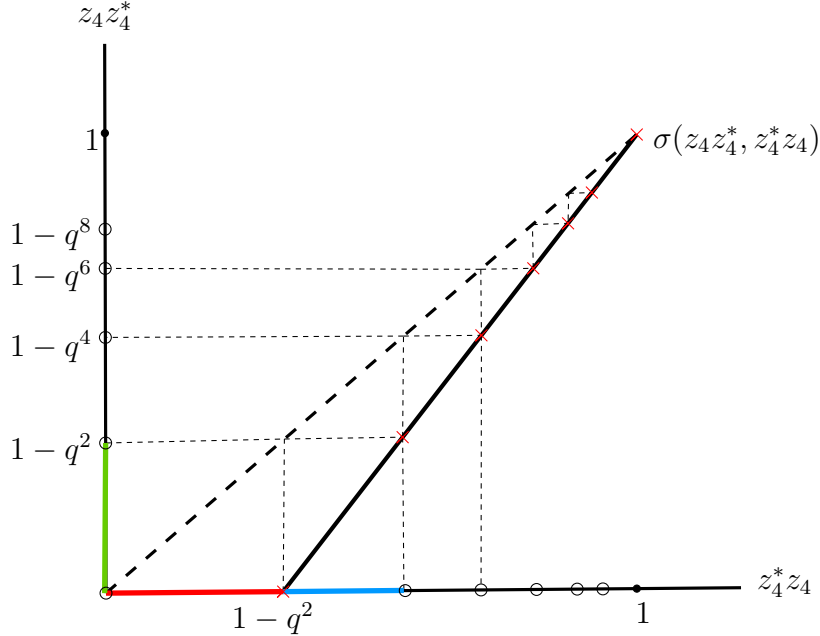
\begin{figure}[H]
\begin{tikzpicture}[scale=0.90]
%\fill[line width=0.pt,color=cqcqcq,fill=cqcqcq,fill opacity=0.550000011920929] (1.1971198468535476,-1.5776791484276884) -- (6.42,5.18) -- (6.420964828141123,-1.5331260995467708) -- cycle;
\draw [line width=1.2pt] (-1.44,6.52)-- (-1.42,-1.6);
\draw [line width=1.2pt] (-1.42,-1.6)-- (7.96,-1.52);
\draw [line width=1.2pt,dash pattern=on 5pt off 5pt] (-1.42,-1.6)-- (6.42,5.18);
\draw [line width=1.6pt] (1.1971198468535476,-1.5776791484276884)-- (6.42,5.18);
\draw (-2.,7.32) node[anchor=north west] {$z_4z_4^*$};
\draw (8.04,-1.06) node[anchor=north west] {$z_4^*z_4$};
\draw (0.28,-1.58) node[anchor=north west] {$1-q^2$};
\draw (-1.9,5.42) node[anchor=north west] {1};
\draw (6.24,-1.58) node[anchor=north west] {1};
\draw [line width=0.4pt,dash pattern=on 2pt off 2pt] (1.1971198468535476,-1.5776791484276884)-- (1.2147115524883585,0.6784877966672278);
\draw [line width=0.4pt,dash pattern=on 2pt off 2pt] (-1.4254675831275745,0.6198387497952522)-- (1.2147115524883585,0.6784877966672278);
\draw [line width=2.pt,color=ffqqqq] (-1.42,-1.6)-- (1.1971198468535476,-1.5776791484276884);
\draw [line width=2.pt,color=wwccqq] (-1.4254675831275745,0.6198387497952522)-- (-1.44,-1.6);
\draw [line width=0.4pt,dash pattern=on 2pt off 2pt] (1.2147115524883585,0.6784877966672278)-- (2.961766958107331,0.7055283230033471);
\draw [line width=0.4pt,dash pattern=on 2pt off 2pt] (2.961766958107331,0.7055283230033471)-- (2.98,-1.56);
\draw [line width=2.pt,color=qqzzff] (1.1971198468535476,-1.5776791484276884)-- (2.98,-1.56);
%\draw [line width=1.2pt] (6.42,5.18)-- (6.4,-1.5);
\draw (6.5,5.5) node[anchor=north west] {$\sigma(z_4z_4^*,z_4^*z_4)$};
\draw [line width=0.4pt,dash pattern=on 2pt off 2pt] (2.961766958107331,0.7055283230033472)-- (2.9743815179552087,2.2002432004765704);
\draw [line width=0.4pt,dash pattern=on 2pt off 2pt] (2.97438151795521,2.2002432004765713)-- (-1.4293105310094214,2.1800755898251003);
\draw [line width=0.4pt,dash pattern=on 2pt off 2pt] (2.97438151795521,2.2002432004765713)-- (4.125485508256192,2.211217971317769);
\draw [line width=0.4pt,dash pattern=on 2pt off 2pt] (4.125485508256192,2.211217971317769)-- (4.119767612068535,-1.552752515035663);
\draw [line width=0.4pt,dash pattern=on 2pt off 2pt] (4.116119738630228,3.1876137535603246)-- (4.125485508256192,2.211217971317769);
\draw [line width=0.4pt,dash pattern=on 2pt off 2pt] (4.116119738630228,3.1876137535603246)-- (-1.4318228310391476,3.2000694018939924);
\draw [line width=0.4pt,dash pattern=on 2pt off 2pt] (4.116119738630228,3.1876137535603246)-- (4.891351513258352,3.202141864023753);
\draw [line width=0.4pt,dash pattern=on 2pt off 2pt] (4.891351513258352,3.202141864023753)-- (4.879029953273638,3.8473753932647017);
\draw [line width=0.4pt,dash pattern=on 2pt off 2pt] (4.879029953273638,3.8473753932647017)-- (5.405154539715735,3.866931385605704);
\draw [line width=0.4pt,dash pattern=on 2pt off 2pt] (5.4140245359941925,4.310036524750078)-- (5.405154539715735,3.866931385605704);
\draw [line width=0.4pt,dash pattern=on 2pt off 2pt] (5.4140245359941925,4.310036524750078)-- (5.757071735201507,4.32226319950434);
\draw (-3,2.44) node[anchor=north west] {$1-q^{4}$};
\draw (-3,3.48) node[anchor=north west] {$1-q^6$};
\draw (-3,4.1) node[anchor=north west] {$1-q^8$};
\draw (-3,0.92) node[anchor=north west] {$1-q^2$};
\begin{scriptsize}
\draw [fill=black] (-1.4367486668648421,5.199958747125949) circle (1.5pt);
\draw [fill=black] (6.420964828141123,-1.5331260995467708) circle (1.5pt);
\draw [color=black] (-1.4254675831275743,0.6198387497952522) circle (2.5pt);
\draw [color=black] (-1.4293105310094214,2.1800755898251003) circle (2.5pt);
\draw [color=black] (3.000190201702905,-1.562301149665647) circle (2.5pt);
\draw [color=ffqqqq] (2.961766958107331,0.7055283230033472)-- ++(-2.5pt,-2.5pt) -- ++(5.0pt,5.0pt) ++(-5.0pt,0) -- ++(5.0pt,-5.0pt);
\draw [color=ffqqqq] (1.1971198468535476,-1.5776791484276884)-- ++(-2.5pt,-2.5pt) -- ++(5.0pt,5.0pt) ++(-5.0pt,0) -- ++(5.0pt,-5.0pt);
\draw [color=ffqqqq] (4.125485508256192,2.211217971317769)-- ++(-2.5pt,-2.5pt) -- ++(5.0pt,5.0pt) ++(-5.0pt,0) -- ++(5.0pt,-5.0pt);
\draw [color=black] (4.119767612068535,-1.552752515035663) circle (2.5pt);
\draw [color=black] (-1.4318228310391476,3.2000694018939924) circle (2.5pt);
\draw [color=ffqqqq] (4.891351513258352,3.202141864023753)-- ++(-2.5pt,-2.5pt) -- ++(5.0pt,5.0pt) ++(-5.0pt,0) -- ++(5.0pt,-5.0pt);
\draw [color=ffqqqq] (5.405154539715735,3.866931385605704)-- ++(-2.5pt,-2.5pt) -- ++(5.0pt,5.0pt) ++(-5.0pt,0) -- ++(5.0pt,-5.0pt);
\draw [color=ffqqqq] (5.757071735201507,4.32226319950434)-- ++(-2.5pt,-2.5pt) -- ++(5.0pt,5.0pt) ++(-5.0pt,0) -- ++(5.0pt,-5.0pt);
\draw [color=black] (4.9398785327556975,-1.5457579656054952) circle (2.5pt);
\draw [color=black] (5.540346581688086,-1.5406367029280335) circle (2.5pt);
\draw [color=black] (5.9198072525764065,-1.53740036458357) circle (2.5pt);
\draw [color=black] (-1.4332512724691664,3.780016622481603) circle (2.5pt);
%\draw [color=black] (-1.4343350097368914,4.220013953177988) circle (2.5pt);
\draw [color=ffqqqq] (6.42,5.18)-- ++(-2.5pt,-2.5pt) -- ++(5.0pt,5.0pt) ++(-5.0pt,0) -- ++(5.0pt,-5.0pt);
\draw [color=black] (-1.42,-1.6) circle (2.5pt);
\end{scriptsize}
\end{tikzpicture}
\caption{Illustration of the joint spectrum $\sigma(z_4z_4^*, z_4^*z_4)$.}
\label{jointspectrum}
\end{figure}

We then obtain 
$$
\sigma(z_4^*z_4)\subseteq \{\lambda_n\mid n\geq 1\}\cup \{1\}, \ \ 
\sigma(z_4z_4^*)\subseteq \{\lambda_n\mid n\geq 0\}\cup \{1\},
$$
where $\lambda_n:=1-q^{2n}$.  In a similar way, we get 
$$
\begin{aligned}
&\sigma(z_3^*z_3)\subseteq \{\delta_{n,m}\mid n\geq 0, m\geq 1\}\cup \{q^{2n}\mid n\geq 1\}\cup \{0\} \\ 
&\sigma(z_3z_3^*)\subseteq \{\delta_{n,m}\mid n,m\geq 0\}\cup \{q^{2n}\mid n\geq 0\}, \\
& \sigma(z_2^*z_2)\subseteq \{\gamma_{n,m,k}\mid n,m\geq 0, k\geq 1\}\cup \{q^{2(m+n)}\mid n,m\geq 0\}\cup \{0\}, \\
&\sigma(z_2z_2^*)\subseteq \{\gamma_{n,m,k}\mid n,m,k \geq 0\} \cup \{q^{2(m+n)}\mid n,m\geq 0\}, \\ 
& \sigma(z_1z_1^*)\subseteq \{q^{2(n+m+k)}\mid n,m,k\geq 0\}\cup \{0\},
\end{aligned}
$$
where $\delta_{n,m}:=q^{2n}(1-q^{2m})$ and $\gamma_{n,m,k}:=(1-q^{2k})q^{2(n+m)}$. 

Let $H_{\lambda_n}$ denote the eigenspace of $\rho(z_4z_4^*)$ for the eigenvalue $\lambda_n$, $H_{\delta_{n,m}}$ the eigenspace of $\rho(z_3z_3^*)$ for the eigenvalue $\delta_{n,m}$, and $H_{\gamma_{n,m,k}}$ the eigenspace of $\rho(z_2z_2^*)$ for the eigenvalue $\gamma_{n,m,k}$. Since $\rho|_{J_1}\neq 0 $ there exists a vector $\xi_{0,0,0}$ such that $\rho(z_1z_1^*)\xi_{0,0,0}\neq 0 $. We will assume that this vector is in $H_{\gamma_{0,0,0}}$ and $\norm{\xi_{0,0,0}}=1$, then $\rho(z_1z_1^*)\xi_{0,0,0}=\xi_{0,0,0}$ and $\rho(z_iz_i^*)\xi_{0,0,0}=0, i=2,3,4$.
We can assume that $\xi_{0,0,0}\in H_{\gamma_{0,0,0}}$ since if $\xi\in H_{\gamma_{n,m,k}}$ is such that $\rho(z_1z_1^*)\xi\neq 0$ then 
$$
\rho({z_4^*}^n{z_3^*}^m{z_2^*}^k{z_1^*}^rz_1^t)\xi\in H_{\gamma_{0,0,0}},
$$
where $t,r$ are chosen so that $k+m+t=r+n$. We can then choose $\xi_{0,0,0}$ to be the above vector after normalization. Indeed, we have 
$$
\begin{aligned}
\rho(z_2z_2^*) & \rho({z_4^*}^n{z_3^*}^m{z_2^*}^k{z_1^*}^rz_1^t)\xi 
= q^{-2(n+m)}\rho({z_4^*}^n{z_3^*}^m z_2{z_2}^*{z_2^*}^k{z_1^*}^rz_1^t)\xi  \\
&= q^{-2(m+n)}\left(\rho({z_4^*}^n{z_3^*}^m {z_2}^*{z_2}{z_2^*}^k{z_1^*}^rz_1^t)
-(1-q^{2})\rho({z_4^*}^n{z_3^*}^m z_1{z_1}^*{z_2^*}^k{z_1^*}^rz_1^t)
\right)\xi \\
&\ \ \vdots 	\\
&= q^{-2(m+n)} \left((1-q^{2k})q^{2(n+m)}- q^{2(n+m+k)}(1-q^2)
(q^{-2k}+q^{-2(k-1)}+\cdots + q^{-2})\right) \\
& \ \ \ \cdot  \rho({z_4^*}^n{z_3^*}^m{z_2^*}^k{z_1^*}^rz_1^t)\xi \\
&= \left((1-q^{2k})-(1-q^2)q^{2k}\frac{q^{-2k}(1-q^{2k})}{1-q^2}\right) \rho({z_4^*}^n{z_3^*}^m{z_2^*}^k{z_1^*}^rz_1^t)\xi
=0. 
\end{aligned} 
$$
Hence $\rho({z_4^*}^n{z_3^*}^m{z_2^*}^k{z_1^*}^rz_1^t)\xi$ is an eigenvector for $\rho(z_2z_2^*)$ with eigenvalue $0$. Similarly, we can show that $\rho({z_4^*}^n{z_3^*}^m{z_2^*}^k{z_1^*}^rz_1^t)\xi$ is an eigenvector for $\rho(z_3z_3^*)$ and $\rho(z_4z_4^*)$ with eigenvalue $0$ for both. We also have 
$$
\begin{aligned} 
\rho(z_1z_1^*)\rho({z_4^*}^n{z_3^*}^m{z_2^*}^k{z_1^*}^rz_1^t)\xi &= q^{-2(n+m+k)} \rho(z_2z_2^*)\rho({z_4^*}^n{z_3^*}^m{z_2^*}^k{z_1^*}^rz_1^tz_1z_1^*)\xi \\
&= q^{-2(n+m+k)}q^{2(n+m+k)}\rho(z_2z_2^*)\rho({z_4^*}^n{z_3^*}^m{z_2^*}^k{z_1^*}^rz_1^t)\xi \\
&= \rho({z_4^*}^n{z_3^*}^m{z_2^*}^k{z_1^*}^rz_1^t)\xi. 
\end{aligned} 
$$
Hence $\rho({z_4^*}^n{z_3^*}^m{z_2^*}^k{z_1^*}^rz_1^t)\xi$ is an eigenvector for $\rho(z_1z_1^*)$ with eigenvalue $1$. We then conclude that 
$\rho({z_4^*}^n{z_3^*}^m{z_2^*}^k{z_1^*}^rz_1^t)\xi\in H_{\gamma_{0,0,0}}$. \\

Now, let 
$$
\xi_{n,m,k}=\rho(z_1^r{z_1^*}^tz_2^kz_3^mz_4^n)\xi_{0,0,0}, \ \text{ for }\, k+m+t=r+n. 
$$
Note that we need to determine $r$ and $t$ such that we can get all combinations of $n,m,k\in \Nb$. There is only one way to choose $r,k$ such that  $z_1^r{z_1^*}^tz_2^kz_3^mz_4^n\in C(\C P_q^{3,\mu})$, and that is why $\xi_{n,m,k}$ is not indexed by $r$ and $t$. 
We want to show that $\xi_{n,m,k}\in H_{\gamma_{n,m,k}}$. 
First we have 
$$
\rho(z_1z_1^*)\xi_{n,m,k}=q^{2(k+m+n)}\xi_{n,m,k} 
$$
and hence $\xi_{n,m,k}$ is an eigenvector for $\rho(z_1z_1^*)$ with eigenvalue $q^{2(k+m+n)}$. We also have to consider $\rho(z_iz_i^*)\text{ for } i=2,3,4$ to conclude that $\xi_{n,m,k}\in H_{\gamma_{n,m,k}}$. For $\rho(z_2z_2^*)$ we get
$$
\begin{aligned}
\rho(z_2z_2^*) & \xi_{n,m,k}
= \rho(z_1^r{z_1^*}^tz_2z_2^*z_2^kz_3^mz_4^n)\xi_{0,0,0} \\
& = \rho(z_1^r{z_1^*}^tz_2^2z_2^*z_2^{k-1}z_3^mz_4^n)\xi_{0,0,0}
+ (1-q^2)\rho(z_1^r{z_1^*}^tz_2z_1z_1^*z_2^{k-1}z_3^mz_4^n)\xi_{0,0,0} \\
&= \rho(z_1^r{z_1^*}^tz_2^2z_2^*z_2^{k-1}z_3^mz_4^n)\xi_{0,0,0}
+ (1-q^2)q^{2(k-1)}q^{2(m+n)}\rho(z_1^r{z_1^*}^tz_2^{k}z_3^mz_4^n)\xi_{0,0,0}\\ 
&= \rho(z_1^r{z_1^*}^tz_2^3z_2^*z_2^{k-2}z_3^mz_4^n)\xi_{0,0,0}
+ (1-q^2)q^{2(m+n)}(q^{2(k-1)}+q^{2(k-2)})\rho(z_1^r{z_1^*}^tz_2^{k}z_3^mz_4^n)\xi_{0,0,0}\\ 
&\ \ \vdots \\
&= \rho(z_1^r{z_1^*}^tz_2^kz_2z_2^*z_3^mz_4^n)\xi_{0,0,0}
\\
&\ \ \ + (1-q^2)q^{2(m+n)}(q^{2(k-1)}+q^{2(k-2)}+\hdots + q^2+1)\rho(z_1^r{z_1^*}^tz_2^{k}z_3^mz_4^n)\xi_{0,0,0} \\
&= (1-q^2)q^{2(m+n)} \frac{q^{2k}-1}{q^2-1} \rho(z_1^r{z_1^*}^tz_2^{k}z_3^mz_4^n)\xi_{0,0,0} \\
&= (1-q^{2k})q^{2(m+n)}\xi_{n,m,k}. 
\end{aligned} 
$$
Hence $\xi_{n,m,k}$ is an eigenvector for $\rho(z_2z_2^*)$ with eigenvalue $\gamma_{n,k,m}$. In a similar way we can show that $\xi_{n,m,k}$ is an eigenvector for $\rho(z_3z_3^*)$ and $\rho(z_4z_4^*)$ with eigenvalues $\delta_{n,m}$ and $\lambda_n$ respectively. Hence $\xi_{n,m,k}\in H_{\gamma_{n,m,k}}$. 

Using the relations for the generators of $C(S_q^7)$, one can show that $J_1\cap\pol(\C P_q^{3,\mu})$ is linearly spanned by monomials of the form
$$
{z_4^*}^{n_4}{z_3^*}^{n_3}{z_2^*}^{n_2}{z_1^*}^{n_1}z_1^{m_1}z_2^{m_2}z_3^{m_3}z_4^{m_4}, 
$$
with $n_2+n_3+m_1+m_4=n_1+n_4+m_2+m_3$, $n_1,m_1\geq 1$. We claim that 
$$
\rho({z_4^*}^{n_4}{z_3^*}^{n_3}{z_2^*}^{n_2}{z_1^*}^{n_1}z_1^{m_1}z_2^{m_2}z_3^{m_3}z_4^{m_4})\xi_{n,m,k}=K(n_i,m_i,n,m,k,r,t)\xi_{ m_4+n-n_4, m_3+m-n_3,m_2+k-n_2 },
$$
where $m_2+k-n_2, m_3+m-n_3, m_4+n-n_4\geqslant 0$ and $K(n_i,m_i,n,m,k,r,t)$ is a numerical value depending on the given parameters. 
To this end, we want to move all terms with $z_4$ and $z_4^*$ to the ends and after that do the same for $z_3$ and $z_2$. 
We illustrate this process by the calculation involving $z_4$. Let 
$$ q':=q^{-m_4(r+t+k+m)}q^{n_4(n_3+n_2+n_1+m_1+m_2+m_3+r+t+k+m)}. $$ 
Then 
$$
\begin{aligned}
&\rho({z_4^*}^{n_4}{z_3^*}^{n_3}{z_2^*}^{n_2}{z_1^*}^{n_1}z_1^{m_1}z_2^{m_2}z_3^{m_3}z_4^{m_4})\xi_{n,m,k}
 \\ 
 &= q'
\rho({z_3^*}^{n_3}{z_2^*}^{n_2}{z_1^*}^{n_1}z_1^{m_1}z_2^{m_2}z_3^{m_3}z_1^{r}{z_1^*}^t z_2^k z_3^m {z_4^*}^{n_4}z_4^{m_4+n})\xi_{0,0,0}  \\
&= q'\rho({z_3^*}^{n_3}{z_2^*}^{n_2}{z_1^*}^{n_1}z_1^{m_1}z_2^{m_2}z_3^{m_3}z_1^{r}{z_1^*}^t z_2^k z_3^m {z_4^*}^{n_4-1}z_4z_4^*z_4^{m_4+n-1})\xi_{0,0,0} 
\\
&\ \ + q'(1-q^2)q^{2(m_4+n-1)}\rho({z_3^*}^{n_3}{z_2^*}^{n_2}{z_1^*}^{n_1}z_1^{m_1}z_2^{m_2}z_3^{m_3}z_1^{r}{z_1^*}^t z_2^k z_3^m {z_4^*}^{n_4-1}z_4^{m_4+n-1}z_1z_1^*)\xi_{0,0,0}    \\
&=  q'\rho({z_3^*}^{n_3}{z_2^*}^{n_2}{z_1^*}^{n_1}z_1^{m_1}z_2^{m_2}z_3^{m_3}z_1^{r}{z_1^*}^t z_2^k z_3^m {z_4^*}^{n_4-1}z_4^2z_4^*z_4^{m_4+n-2})\xi_{0,0,0} 
\\ 
&\ \ + q'(1-q^2)q^{2(m_4+n-2)}\rho({z_3^*}^{n_3}{z_2^*}^{n_2}{z_1^*}^{n_1}z_1^{m_1}z_2^{m_2}z_3^{m_3}z_1^{r}{z_1^*}^t z_2^k z_3^m {z_4^*}^{n_4-1}z_4^{m_4+n-1})\xi_{0,0,0}    \\
&\ \ + q'(1-q^2)q^{2(m_4+n-1)}\rho({z_3^*}^{n_3}{z_2^*}^{n_2}{z_1^*}^{n_1}z_1^{m_1}z_2^{m_2}z_3^{m_3}z_1^{r}{z_1^*}^t z_2^k z_3^m {z_4^*}^{n_4-1}z_4^{m_4+n-1})\xi_{0,0,0}    
\\
& \ \ \vdots
\\
&= q'(1-q^2)(q^{2(m_4+n-1)}+\hdots + q^2+1)\rho({z_3^*}^{n_3}{z_2^*}^{n_2}{z_1^*}^{n_1}z_1^{m_1}z_2^{m_2}z_3^{m_3}z_1^{r}{z_1^*}^t z_2^k z_3^m {z_4^*}^{n_4-1}z_4^{m_4+n-1})\xi_{0,0,0}  \\
&= q'(1-q^{2(m_4+n)})\rho({z_3^*}^{n_3}{z_2^*}^{n_2}{z_1^*}^{n_1}z_1^{m_1}z_2^{m_2}z_3^{m_3}z_1^{r}{z_1^*}^t z_2^k z_3^m {z_4^*}^{n_4-1}z_4^{m_4+n-1})\xi_{0,0,0}.
\end{aligned}
$$
Continuing as above we get
$$
\begin{aligned}
&\rho({z_4^*}^{n_4}{z_3^*}^{n_3}{z_2^*}^{n_2}{z_1^*}^{n_1}z_1^{m_1}z_2^{m_2}z_3^{m_3}z_4^{m_4})\xi_{n,m,k} \\
&=q'(1-q^{2(m_4+n)})(1-q^{2(m_4+n-1)})(1-q^{2(m_4+n-2)})\cdots(1-q^{2(m_4+n-n_4+1)})
\\
&\ \ \ \ \rho({z_3^*}^{n_3}{z_2^*}^{n_2}{z_1^*}^{n_1}z_1^{m_1}z_2^{m_2}z_3^{m_3}z_1^{r}{z_1^*}^t z_2^k z_3^m z_4^{m_4+n-n_4})\xi_{0,0,0}.
\end{aligned}
$$
Let 
$$
\begin{aligned}
K_1(n_i,m_i,n,m,k,r,t)&= q^{-m_4(r+t+k+m)}q^{n_4(n_3+n_2+n_1+m_1+m_2+m_3+r+t+k+m)}\prod_{i=0}^{n_4-1} (1-q^{2(m_4+n-i)}) \\
K_2(n_i,m_i,n,m,k,r,t)&=q^{-m_3(r+t+m)}q^{n_3(n_2+n_1+m_1+m_2+r+t+k)} q^{2n_3(m_4+n-n_4)} \prod_{i=0}^{n_3-1} (1-q^{2(m_3+m-i)})\\
K_3(n_i,m_i,n,m,k,r,t)&=q^{-m_2(r+t)}q^{n_2(n_1+m_1)}q^{2n_2(m_4+m-n_4+m_3+m-n_3)} \prod_{i=0}^{n_2-1} (1-q^{2(m_2+k-i)})
\end{aligned} 
$$
and denote the product 
$$ K_1(n_i,m_i,n,m,k,r,t)K_2(n_i,m_i,n,m,k,r,t)K_3(n_i,m_i,n,m,k,r,t) $$
by $K(n_i,m_i,n,m,k,r,t)$. Then 
\begin{equation}\label{rep}
\rho({z_4^*}^{n_4}{z_3^*}^{n_3}{z_2^*}^{n_2}{z_1^*}^{n_1}z_1^{m_1}z_2^{m_2}z_3^{m_3}z_4^{m_4})\xi_{n,m,k}=K(n_i,m_i,n,m,k,r,t)\xi_{m_4+n-n_4, m_3+m-n_3,m_2+k-n_2 }. 
\end{equation}

If $m_2+k< n_2$, $m_3+m<n_3$ or $m_4+n<n_4$ then 
$$ \rho({z_4^*}^{n_4}{z_3^*}^{n_3}{z_2^*}^{n_2}{z_1^*}^{n_1}z_1^{m_1}z_2^{m_2}z_3^{m_3}z_4^{m_4})\xi_{n,m,k}=0. $$ 
Indeed 
$$
\begin{aligned}
&\norm{\rho({z_4^*}^{n_4}{z_3^*}^{n_3}{z_2^*}^{n_2}{z_1^*}^{n_1}z_1^{m_1}z_2^{m_2}z_3^{m_3}z_4^{m_4})\xi_{n,m,k}}^2
\\
&= \inpro{\xi_{n,m,k}}{\rho(
{z_4^*}^{m_4}{z_3^*}^{m_3}{z_2^*}^{m_2}{z_1^*}^{m_1}z_1^{n_1}z_2^{n_2}z_3^{n_3}z_4^{n_4}
{z_4^*}^{n_4}{z_3^*}^{n_3}{z_2^*}^{n_2}{z_1^*}^{n_1}z_1^{m_1}z_2^{m_2}z_3^{m_3}z_4^{m_4})\xi_{n,m,k}}.
\end{aligned}
$$
Moving $z_4$ and $z_4^*$ to the ends we get the following expression up to a constant
$$
\rho({z_3^*}^{m_3}{z_2^*}^{m_2}{z_1^*}^{m_1}z_1^{n_1}z_2^{n_2}z_3^{n_3}
{z_3^*}^{n_3}{z_2^*}^{n_2}{z_1^*}^{n_1}z_1^{m_1}z_2^{m_2}z_3^{m_3}z_2^kz_3^m{z_4^*}^{m_4}{z_4}^{n_4}{z_4^*}^{n_4}{z_4}^{m_4+n})\xi_{0,0,0}.
$$
If $n_4>m_4+n$ then performing a calculation similar to the above we get the following, up to some constant
$$
\rho({z_3^*}^{m_3}{z_2^*}^{m_2}{z_1^*}^{m_1}z_1^{n_1}z_2^{n_2}z_3^{n_3}
{z_3^*}^{n_3}{z_2^*}^{n_2}{z_1^*}^{n_1}z_1^{m_1}z_2^{m_2}z_3^{m_3}z_2^kz_3^m{z_4}^{n_4-m_4}{z_4^*}^{n_4-(m_4+n)})\xi_{0,0,0}
$$
We can now reduce this to the following up to a constant 
$$
\rho({z_3^*}^{m_3}{z_2^*}^{m_2}{z_1^*}^{m_1}z_1^{n_1}z_2^{n_2}z_3^{n_3}
{z_3^*}^{n_3}{z_2^*}^{n_2}{z_1^*}^{n_1}z_1^{m_1}z_2^{m_2}z_3^{m_3}z_2^kz_3^m{z_4}^{n+1}{z_4^*})\xi_{0,0,0}=0. 
$$
A similar calculation can be done when $m_2+k< n_2$ or $m_3+m<n_3$. 
\\

Note that  $H_{\gamma_{0,0,0}}$ is one dimensional since $\rho$ is irreducible. Indeed, if $\eta_{0,0,0}\in H_{\gamma_{0,0,0}}$ and is orthogonal to $\xi_{0,0,0}$ then we can use the same construction above for $\eta_{0,0,0}$ instead of $\xi_{0,0,0}$. By the construction we get that  $\rho(C(\CP_q^{3,\mu}))\xi_{0,0,0}$ and $\rho(C(\CP_q^{3,\mu}))\eta_{0,0,0}$ are orthogonal dense subspaces of $H$ which cannot be true. \\

We now want to normalize $\xi_{n,m,k}$. The norm is given by the following using equation (\ref{rep})
$$
\begin{aligned}
\norm{\xi_{n,m,k}}^2 &=\inpro{\xi_{0,0,0}}{\rho({z_4^*}^n{z_3^*}^m{z_2^*}^k{z_1^*}^{r+t}z_1^{r+t}z_2^kz_3^mz_4^n)\xi_{0,0,0}} \\
&= q^{n(2(m+k)+r+t)}q^{m(2k+r+t)} \prod_{i=1}^n (1-q^{2i})\prod_{i=1}^m (1-q^{2i})\prod_{i=1}^k (1-q^{2i})=:N_{n,m,k}. 
\end{aligned}
$$

With the normalization constant we get a basis for our Hilbert space, also denoted $\xi_{n,m,k}$, given by 
$$
\xi_{n,m,k}:={N_{n,m,k}}^{-1/2}\rho(z_1^r{z_1^*}^tz_2^kz_3^mz_4^n)\xi_{0,0,0},  
$$
with $k+m+t=r+n$ and $n,m,k,r,t\geqslant 0$. 
\\

Hence if $\rho$ is an irreducible representation of $C(\C P_q^{3,\mu}))$ then $\rho$ restricted to the ideal $J_1$ is given by the following for $n_2+n_3+m_1+m_4=n_1+n_4+m_2+m_3$, $n_1,m_1\geq 1$
$$
\begin{aligned}
\rho({z_4^*}^{n_4}{z_3^*}^{n_3}{z_2^*}^{n_2} & {z_1^*}^{n_1}z_1^{m_1}z_2^{m_2}z_3^{m_3}z_4^{m_4})\xi_{n,m,k}=  C(n_i,m_i,n,m,k,r,t)\xi_{m_4+n-n_4, m_3+m-n_3, m_2+k-n_2},
\end{aligned}
$$
when $m_2+k-n_2, m_3+m-n_3, m_4+n-n_4\geqslant 0$ otherwise it is $0$. Here 
$$
\begin{aligned}
C(n_i,m_i,n,m,k,r,t)&=
K(n_i,m_i,n,m,k,r,t)
q^{(m_3+m-n_3)2(n_2-m_2)+(n_3-m_3)(2k+r+t)}
\\
&\cdot (q^{(m_4+n-n_4)2(n_3-m_3+n_2-m_2)+(2(m+k)+r+t)(n_4-m_4)})^{-1/2}
\\
&\cdot\left( \prod_{i=m_4+n-n_4+1}^n (1-q^{2i})\prod_{i=m_4+m-n_3+1}^m (1-q^{2i}) \prod_{i=m_2+k-n_2+1}^k (1-q^{2i})\right)^{-1/2}
\end{aligned} 
$$
where the last part takes the normalisation into account . 

Since all irreducible representations of $J_1$ can be extended to $C(\C P_q^{3,\mu})$, \cite[Theorem 5.5.1]{gm}, and by the result above there only exists one irreducible representation, up to unitary equivalence, of $J_1$, hence $J_1$ is compact.

Note that $\rho(z_1z_1^*)$ is trace class since 
$$
\sum_{n,m,k\geq 0}\inpro{\rho(z_1z_1^*)\xi_{n,m,k}}{\xi_{n,m,k}}=\sum_{n,m,k\geq 0} q^{2(n+m+k)} 
$$
which is finite because $q\in (0,1)$. Then $\rho(J_1)\subseteq K(H)$. By \cite[Theorem 10.4.10]{kr} we get $K(H)\subseteq \rho(J_1)$ hence $\rho(J_1)=K(H)$. 
\end{proof}

We then obtain the following short exact sequence of $C^*$-algebras 
\begin{equation}\label{exactseq}
0\to \mathcal{K} \to C(\C P_q^{3,\mu}) \to C^*(F) \to 0. 
\end{equation}
The space of compact operators is an AF algebra. Moreover, $C^*(F)$ is AF since it does not contain any cycles, \cite[Corollary 2.13]{dt}. Then $C(\C P_q^{3,\mu})$ is AF since the ideal and the quotient in \eqref{exactseq} are both AF algebras. We note that  the ideal $\mathcal{K}$ is essential in $C(\C P_q^{3,\mu})$, since it intersects non-trivially each non-zero ideal of $C(\C P_q^{3,\mu})$. 

From the six term exact sequence of $K$-theory we get the following short exact sequence 
$$
0\to K_0(\mathcal{K})\to K_0(C(\C P_q^{3,\mu}))\to K_0(C^*(F))\to 0
$$
where $K_0(\mathcal{K})\cong \Z$ and $K_0(C^*(F))\cong \Z^3$ by \cite{hs}. The above short exact sequence splits since $\Z^3$ is a free module over $\Z$ and hence projective. Thus $K_0(C(\C P_q^{3,\mu})))\cong \Z\oplus \Z^3 \cong \Z^4.$

\begin{theorem}
The $C^*$-algebra $C(\C P_q^{3,\mu})$ is isomorphic to the graph $C^*$-algebra $C^*(G)$ for the following graph. 
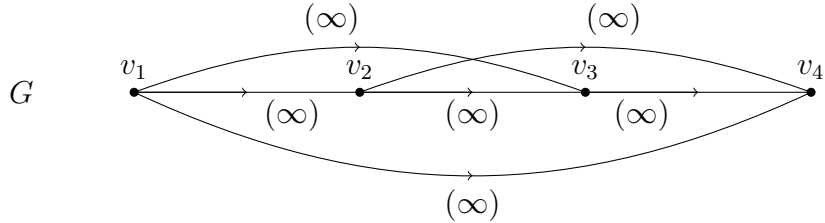
\begin{center}
\begin{tikzpicture}[scale=1.5]
\draw (0,0) -- (2,0);
\draw  (2,0) -- (4,0);
\draw (4,0) -- (6,0);
\filldraw [black] (0,0) circle (1pt);
\filldraw [black] (2,0) circle (1pt);
\filldraw [black] (4,0) circle (1pt);
\filldraw [black] (6,0) circle (1pt);
\draw [->] (0,0) -- (1,0);
\draw [->] (4,0) -- (5,0);
\draw [->] (2,0) -- (3,0);

\draw[] (0,0) to [out=20,in=160] (4,0);
\draw [->] (1.99,0.4) -- (2,0.4);
\draw[] (2,0) to [out=20,in=160] (6,0);
\draw [->] (3.99,0.4) -- (4,0.4);
\draw[] (0,0) to [out=-25,in=-155] (6,0);
\draw [->] (2.99,-0.74) -- (3,-0.74);

\node at (-1, 0)  {$G$};
\node at (0, 0.2)  {$v_1$};
\node at (2, 0.2)  {$v_2$};
\node at (4, 0.2)  {$v_{3}$};
\node at (6, 0.2)  {$v_4$};
\node at (1.4, -0.2)  {$(\infty)$};
\node at (4.5, -0.2)  {$(\infty)$};
\node at (3, -0.2)  {$(\infty)$};
\node at (1.75, 0.65)  {$(\infty)$};
\node at (4.25, 0.65)  {$(\infty)$};
\node at (3, -1)  {$(\infty)$};
\end{tikzpicture}
\end{center}
\end{theorem}
\begin{proof}
The K-groups of $C^*(G)$ are $K_0(C^*(G))\cong \Z^4$ and $K_1(C^*(G))=0$, \cite{dt1,hs}.  Hence $C(\C P_q^{3,\mu})$ and $C^*(G)$ have the the same $K_0$-groups and they are both AF algebras. We want to use Elliott's classification theorem for unital AF algebras \cite{ge} , that is, we wish to show that the triples $(K_0(C(\C P_q^{3,\mu})), K_0(C(\C P_q^{3,\mu}))^+, [1]_0)$ and $(K_0(C^*(G)), K_0(C^*(G))^+, [1]_0)$ are isomorphic as ordered groups. First we will show that $K_0(C(\C P_q^{3,\mu}))^+$ and $K_0(C^*(G))^+$ are the same. 

We denote by $\{s_e,p_v| e\in G^1, v\in G^0\}$ the Cuntz-Krieger family defining $C^*(G)$. A generating set for $K_0(C^*(G))$ is $[1]_0, [p_{v_2}+p_{v_3}+p_{v_4}]_0, [p_{v_3}+p_{v_4}]_0$ and $[p_{v_4}]_0$. We see that we can subtract the next generator from the previous one a number of times and still get a projection. Indeed, let $S_{v}$ be a finite subset of $s^{-1}(v)$ then 
$$
\begin{aligned} 
\ [p_v]_0&=[p_v-\sum_{e\in S_v}s_es_e^*]_0+\sum_{e\in S_v}[s_es_e^*]_0=[p_v-\sum_{e\in S_v}s_es_e^*]_0+\sum_{e\in S_v}[s_e^*s_e]_0 \\ &=
[p_v-\sum_{e\in S_v}s_es_e^*]_0+\sum_{e\in S_v}[p_{r(e)}]_0.
\end{aligned} 
$$ 
Hence 
$$
[p_v]_0-\sum_{e\in S_v}[p_{r(e)}]_0=[p_v-\sum_{e\in S_v}s_es_e^*]_0\in K_0(C^*(G))^+.
$$
Consider the following calculation which follows by the above 
$$
\begin{aligned} 
\ [1]_0-3[p_{v_2}+p_{v_3}+p_{v_4}]_0&= [p_{v_1}]-2[p_{v_2}+p_{v_3}+p_{v_4}]_0 \\
&= [p_{v_1}-\sum_{i=2,3,4} s_{e_i}s_{e_i}^*]_0-[p_{v_2}+p_{v_3}+p_{v_4}]_0 \\
&= [p_{v_1}-\sum_{i=2,3,4} s_{e_i}s_{e_i}^*]_0- [\sum_{i=2,3,4} s_{f_i}^*s_{f_i}]  \\
&= [p_{v_1}-\sum_{i=2,3,4} s_{e_i}s_{e_i}^*]_0- \sum_{i=2,3,4} [s_{f_i}s_{f_i}^*]  \\
&= [p_{v_1}-\sum_{i=2,3,4} s_{e_i}s_{e_i}^*-\sum_{i=2,3,4} s_{f_i}s_{f_i}^*]_0
\end{aligned} 
$$
where $e_i,f_i\in r^{-1}(v_i)\cap s^{-1}(v_1), i=2,3,4$. 
Since $s^{-1}(v_i)\cap s^{-1}(v_1), i=2,3,4$ consists of infinitely many edges we can subtract $n[p_{v_2}+p_{v_3}+p_{v_4}]_0$ from $[1]_0$ for all $n\in\Nb$ and still get an element in $K_0(C^*(G))^+$. By similar calculations we see that we can subtract the next generator from the previous one a number of times and still get an element in $K_0(C^*(G))^+$. 

Identifying $[1]_0$ with $(1,0,0,0)$,  $[p_{v_2}+p_{v_3}+p_{v_4}]_0$ with $(0,1,0,0)$ and so on, we get that an element $(n_1,n_2,n_3,n_4)$ is positive if and only if we have one of the following four cases: 
$$
\begin{aligned} 
&\text{(1)} \ n_i\geq 0, i=1,2,3,4, \ \ \ &&\text{(2)} \ n_1\geq 1 \  \text{and} \ n_i\in \Z \ \text{for} \ i=2,3,4,
\\
&\text{(3)} \ n_1\geq 0, n_2\geq 1 \ \text{and} \ n_3,n_4\in \Z, \ \ \ \  
&&\text{(4)} \ n_1\geq 0, n_2\geq 0, n_3\geq 1  \ \text{and} \ n_4\in \Z.
\end{aligned}
$$

We know that $K_0(C(\C P_q^{3,\mu})) \cong K_0(\mathcal{K})\oplus K_0(C^*(F))$. A generating set for $K_0(C^*(F))$ is $[1]_0, [p_{v_2}+p_{v_3}]_0$ and $[p_{v_3}]_0$. Since the ideal $J_1$ is AF all the projections in the quotient lift to a projection, \cite[Lemma 9.7]{ee}, hence there exists projections $q_1,q_2$ and $q_3$ in $C(\C P_q^{3,\mu})$ such that 
$$
[\pi(q_1)]_0=[1]_0, \ [\pi(q_2)]_0=[p_{v_2}+p_{v_3}]_0, \ [\pi(q_3)]_0=[p_{v_3}]_0. 
$$
where $\pi$ is the quotient map. Let $q_4$ be the minimal projection in $\mathcal{K}$, then the last generator is $[q_4]_0$. 

Consider an element $n_1[q_1]_0+n_2[q_2]_0+n_3[q_3]_0+n_4[q_4]_0\in K_0(C(\C P_q^{3,\mu}))$. We wish to determine for which coefficients $n_i\in \Z, i=1,2,3,4$ the element is in $K_0(C(\C P_q^{3,\mu}))^+$. 

We have 
$$
K_0(\pi)(n_1[q_1]_0+n_2[q_2]_0+n_3[q_3]_0+n_4[q_4]_0)=n_1[1]_0+n_2[p_{v_2}+p_{v_3}]_0+n_3[p_{v_3}]_0.
$$
As for $C^*(G)$ we have that  $n_1[1]_0+n_2[p_{v_2}+p_{v_3}]_0+n_3[p_{v_3}]_0\in K_0(C^*(F))^+$ if and only if one of the following three cases are satisfied: 
$$
\begin{aligned} 
\text{(1)} \ n_i\geq 0, i=1,2,3, \ \ \ \ \text{(2)} \ n_1\geq 1 \ \text{and} \ n_2,n_3\in \Z, \ \ \ \ \text{(3)} \ n_1\geq 0, n_2\geq 1 \ \text{and} \ n_3\in \Z.
\end{aligned}  
$$
Hence 
$$
K_0(\pi)(n_1[q_1]_0-n_2[q_2]_0-n_3[q_3]_0-n_4[q_4]_0)=K_0(\pi)([p]_0)
$$
for a projection $p\in C(\C P_q^{3,\mu})$ if and only if the coefficients satisfies one of the above conditions. Then there must exist an $m_4\in \Z$ such that 
$$
n_1[q_1]_0+n_2[q_2]_0+n_3[q_3]_0+(m_4+n_4)[q_4]_0= [p]_0 \in  K_0(C(\C P_q^{3,\mu}))^+. 
$$
Identifying $[q_i]_0$ with $(0,..,1,...0)$ where $1$ is in the $i'$th entry we get that $K_0(C(\C P_q^{3,\mu}))^+$ and $K_0(C^*(F))^+$ are the same. 

The unit $[1]_0$ in both cases corresponds to $(1,0,0,0)$. By Elliot's classification theorem we have that $K_0(C(\C P_q^{3,\mu}))$ is isomorphic to $C^*(G)$. 

\end{proof}

%%%%%%%%%%%%%%%%%%%%%%%%%%%%%%%%%%%%%%%%%%%%%%%%%%%%%%%%%%%%%
%%%%%%%%%%%%%%%%%%%%%%%%%%%%%%%%%%%%%%%%%%%%%%%%%%%%%%%%%%%%%

\subsection{Freenes of the circle action}\label{Freenes}
The circle action on $C(S_q^7)$ defining $C(\CP_q^{3,\mu})$ is slightly different from the gauge action but, as Theorem \ref{TheCalgebra} proves, their fixed point algebras are the same. In this section we show that the modified action $\mu$ is free, which is also the case for the gauge action due to \cite[Proposition 2]{S01}. 

We will consider the graph $C^*$-algebraic picture of $C(S_q^7)$ to show that the action $\mu$ is free. Under the isomorphism $C(S_q^7)\cong C^*(L_7)$ the action $\mu$ becomes 
$$
\mu_w(S_{e_{ij}})=\begin{cases}
wS_{e_{ij}} & i=1,4 \\
\overline{w}S_{e_{ij}} & i=2,3
\end{cases}, \ \ \ \mu_w(P_{v_i})=P_{v_i}, i=1,2,3,4
$$
for all $w\in U(1)$. Then the corresponding coaction $\hat{\mu}:C^*(L_7)\to C^*(L_7)\otimes C(S^1) $ of $\mu$ is given on the generators as follows: 
$$
\hat{\mu}(P_{v_i})=P_{v_i}, \ \ 
\hat{\mu}(S_{e_{ij}})=\begin{cases}
S_{e_{ij}}\otimes u, & i=1,4 \\
S_{e_{ij}}\otimes u^*, & i=2,3
\end{cases}
$$
where $u$ is the canonical generator of $C(S^1)$. 

\begin{proposition} 
The action $\mu$ defined in \eqref{newaction} is free.  
\end{proposition}
\begin{proof}
It suffices to show that the image of the map $\Phi$ (see Definition \ref{Free}) contains $P_{v_i}\otimes u^k$ and $P_{v_i}\otimes {u^*}^k$ for $i=1,...,4$ and any $k\in \mathbb{N}$. 

First we have immediately for all $k\in\mathbb{N}$ the following:  
$$
\begin{aligned}
\Phi({S_{e_{ii}}^*}^k\otimes S_{e_{ii}}^k)=\begin{cases}
({S_{e_{ii}}^*}^k\otimes 1)(S_{e_{ii}}^k\otimes u^k), & i=1,4\\
({S_{e_{ii}}^*}^k\otimes 1)(S_{e_{ii}}^k\otimes {u^*}^k), & i=2,3 
\end{cases}
= \begin{cases}
P_{v_i}\otimes u^k, & i=1,4 \\
P_{v_i}\otimes {u^*}^k, & i=2,3
\end{cases}.
\end{aligned}
$$
Hence the image contains $P_{v_i}\otimes u^k, i=1,4$ and $P_{v_j}\otimes {u^*}^k, j=2,3$ for any $k\in \mathbb{N}$. 

To show that the image of $\Phi$ contains $P_{v_i}\otimes u^k$ for  $i=2,3$ and $P_{v_j}\otimes {u^*}^k$ for $j=1,4$ for any $k\in \mathbb{N}$ we need to consider the four cases separately. First note that since $P_{v_4}=S_{e_{44}}S_{e_{44}}^*$ we have 
$$
\Phi(S_{e_{44}}^k\otimes {S_{e_{44}}^*}^k)= S_{e_{44}}^k{S_{e_{44}}^*}^k\otimes {u^*}^k=P_{v_4}\otimes {u^*}^k$$
for any $k\in\mathbb{N}$. 

For $i=2$ we follow a similar approach as the one in \cite[Proposition 2]{S01}, and apply that for any $k\in\mathbb{N}$ we have 
$$
P_{v_2}=\sum_{s(\beta)=v_2, |\beta|=k} S_{\beta}S_{\beta}^*
$$
where $\beta$ is a path of length $k$ starting in $v_2$. 

Let $B$ be the set of all paths of length $k$ with source $v_2$. Let $B_1$ be the paths $\beta\in B$ such that $\beta=f_1f_2\cdots f_k$ with $s(f_i)\in \{v_2,v_3\}$ for $i=1,..,k$. Then for any $\beta\in B_1$ and $k\in\mathbb{N}$ we have
$$
\Phi(S_{\beta}\otimes S_{\beta}^*)=S_{\beta}S_{\beta}^*\otimes u^k.
$$
On the other hand for $\beta\in B\setminus B_1$ the above might not be the case. The paths in $B\setminus B_1$ are of the form 
\begin{itemize}
\item[(1)] $\beta_{l,m}=e_{22}^{k-m-l-2}e_{23}e_{33}^me_{34}e_{44}^l$ with $l\geq 1$ and $m\geq 0$,
\item[(2)] $\beta_l=e_{22}^{k-l-1}e_{24}e_{44}^l$ with $l\geq 1$,
\end{itemize}
where $e_{ii}^k$ is the path going through the loop $e_{ii}$ $k$ times. We denote by $B_2$ all paths of type (1) and by $B_3$ all paths of type (2). 

We have 
$$
S_{\beta_{l,m}}S_{\beta_{l,m}}^*=S_{\beta_{l,m}}P_{v_4}S_{\beta_{l,m}}^*=S_{\beta_{l,m}}S_{{\alpha}_{t}}^*S_{{\alpha}_{t}}S_{\beta_{l,m}}^*
$$
for $\alpha_t=e_{44}^{t}, t\in\mathbb{N}$ since $r(\alpha_t)=v_4$ and $r(\beta_{l,m})=v_4$. Then 
$$
\Phi(S_{{\beta}_{l,m}}S_{{\alpha}_{2l}}^*\otimes S_{{\alpha}_{2l}}S_{{\beta}_{l,m}}^*)=S_{{\beta}_{l,m}}S_{{\alpha}_{2l}}^*S_{{\alpha}_{2l}}S_{{\beta}_{l,m}}^*\otimes u^{k-l}{u^*}^lu^{2l}=S_{\beta_{l,m}}S_{\beta_{l,m}}^*\otimes u^k.
$$
Hence we obtain that 
$$
\begin{aligned}
&\Phi\left(\sum_{\beta\in B_1} S_{\beta}\otimes S_\beta^* +  \sum_{\beta\in B_2} S_{{\beta}_{l,m}}S_{{\alpha}_{2l}}^*\otimes S_{{\alpha}_{2l}}S_{{\beta}_{l,m}}^* + \sum_{\beta\in B_3} S_{{\beta}_{l}}S_{{\alpha}_l}^*\otimes S_{{\alpha}_l}S_{{\beta}_{l}}^* \right) \\
& \ \ = \sum_{\beta\in B} S_\beta{S_\beta}^*\otimes u^k = P_{v_2}\otimes u^k. 
\end{aligned}
$$
The calculations for $i=3,4$ follows by a similar approach, where the one for $i=3$ is an easier computation and the one for $i=4$ requires more but similar calculations. 
\end{proof}

\section{Projections in $\mathcal{O}(\C P_q^{3,\mu})$}\label{projections}

%%%%%%%%%%%%%%%%%%%%%%%%%%%%%%%%%%%%%%%%%%%%%%%%%%%%%%%%%%%%%
%%%%%%%%%%%%%%%%%%%%%%%%%%%%%%%%%%%%%%%%%%%%%%%%%%%%%%%%%%%%%

In applications, it is useful to have elements of $K$-theory represented by projections not merely in matrices over the $C^*$-algebra 
$C(\C P_q^{3,\mu})$ but by matrices with entries in the polynomial  algebra $\mathcal{O}(\C P_q^{3,\mu})$. 

Following the approach of 
\cite[Lemma 3.2]{DAnLan:ant}, we give  a construction of projections in matrices over $\mathcal{O}(\C P_q^{3,\mu})$. The projections are constructed 
in a recursive way. We will explicitly write down two of them. These two non trivial projections will serve as candidates for the generators of the K-theory together with the  trivial projection and the non-trivial polynomial generator of the K-theory of $C(S_q^4)$, $G$, given by \cite[Proposition 7]{bct1}
$$
G=\begin{pmatrix}
q^2R & 0 & qa & q^2b \\ 
0 & q^2R & qb^* & -q^3a^* \\ 
qa^* & qb & 1-R & 0 \\
q^2b^* & -q^3 a & 0 & 1-q^4R
\end{pmatrix}.
$$

More specifically,  for $N\in \Nb$ let 
$$
\psi_{j_1,j_2,j_3,j_4}^N= \sqrt{c_{j_1,j_2,j_3,j_4}(N)}{z_4^*}^{j_4}z_3^{j_3}z_2^{j_2}{z_1^*}^{j_1}  
$$
where $j_1+j_2+j_3+j_4=N$. Let $A^N$ be the square matrix with $\psi_{j_1,j_2,j_3,j_4}^N$ as the elements in the first column and zero elsewhere. Note that  $({A^N}^*A^N)_{ij}=0$ for $i,j\neq 1$. We wish to find coefficients $c_{j_1,j_2,j_3,j_4}(N)$ such that $({A^N}^*A^N)_{11}=1$. Then $P_N=A^N{A^N}^*$ will be a projection in matrices over $\mathcal{O}(\C P_q^{3,\mu})$ for all $N\in\Nb$.

By the condition $({A^N}^*A^N)_{11}=1$ and using that $q^4z_1z_1^*+q^2z_2^*z_2+z_3^*z_3+z_4z_4^*=1$ we have: 
$$
\begin{aligned}
1&=\sum_{j_1+j_2+j_3+j_4=N+1}c_{j_1,j_2,j_3,j_4}(N+1) (z_1^{j_1}{z_2^*}^{j_2}{z_3^*}^{j_3}z_4^{j_4})(z_1^{j_1}{z_2^*}^{j_2}{z_3^*}^{j_3}z_4^{j_4})^* \\
&= \sum_{j_1+j_2+j_3+j_4=N+1}c_{j_1,j_2,j_3,j_4}(N+1) (z_1^{j_1}{z_2^*}^{j_2}{z_3^*}^{j_3}z_4^{j_4})(q^4z_1z_1^*+q^2z_2^*z_2+z_3^*z_3+z_4z_4^*)(z_1^{j_1}{z_2^*}^{j_2}{z_3^*}^{j_3}z_4^{j_4})^*
\\
&= \sum_{j_1+j_2+j_3+j_4=N+1}c_{j_1,j_2,j_3,j_4}(N+1) q^4q^{2(j_2+j_3-j_4)} 
(z_1^{j_1+1}{z_2^*}^{j_2}{z_3^*}^{j_3}z_4^{j_4})(z_1^{j_1+1}{z_2^*}^{j_2}{z_3^*}^{j_3}z_4^{j_4})^* \\
&+ \sum_{j_1+j_2+j_3+j_4=N+1}c_{j_1,j_2,j_3,j_4}(N+1) q^2q^{2(j_3-j_4)}
(z_1^{j_1}{z_2^*}^{j_2+1}{z_3^*}^{j_3}z_4^{j_4})(z_1^{j_1}{z_2^*}^{j_2+1}{z_3^*}^{j_3}z_4^{j_4})^* 
\\
&+ \sum_{j_1+j_2+j_3+j_4=N+1}c_{j_1,j_2,j_3,j_4}(N+1) q^{-2j_4}
(z_1^{j_1}{z_2^*}^{j_2}{z_3^*}^{j_3+1}z_4^{j_4})(z_1^{j_1}{z_2^*}^{j_2}{z_3^*}^{j_3+1}z_4^{j_4})^* \\
&+ \sum_{j_1+j_2+j_3+j_4=N+1}c_{j_1,j_2,j_3,j_4}(N+1)
(z_1^{j_1}{z_2^*}^{j_2}{z_3^*}^{j_3}z_4^{j_4+1})(z_1^{j_1}{z_2^*}^{j_2}{z_3^*}^{j_3}z_4^{j_4+1})^* 
\\
&= \sum_{j_1-1+j_2+j_3+j_4=N}c_{j_1-1,j_2,j_3,j_4}(N) q^4q^{2(j_2+j_3-j_4)} 
(z_1^{j_1}{z_2^*}^{j_2}{z_3^*}^{j_3}z_4^{j_4})(z_1^{j_1}{z_2^*}^{j_2}{z_3^*}^{j_3}z_4^{j_4})^* \\
&+ \sum_{j_1+j_2-1+j_3+j_4=N}c_{j_1,j_2-1,j_3,j_4}(N) q^2q^{2(j_3-j_4)}
(z_1^{j_1}{z_2^*}^{j_2}{z_3^*}^{j_3}z_4^{j_4})(z_1^{j_1}{z_2^*}^{j_2}{z_3^*}^{j_3}z_4^{j_4})^* 
\\
&+ \sum_{j_1+j_2-1+j_3+j_4=N}c_{j_1,j_2,j_3-1,j_4}(N) q^{-2j_4}
(z_1^{j_1}{z_2^*}^{j_2}{z_3^*}^{j_3}z_4^{j_4})(z_1^{j_1}{z_2^*}^{j_2}{z_3^*}^{j_3}z_4^{j_4})^* \\
&+ \sum_{j_1+j_2+j_3+j_4-1=N}c_{j_1,j_2,j_3,j_4-1}(N)
(z_1^{j_1}{z_2^*}^{j_2}{z_3^*}^{j_3}z_4^{j_4})(z_1^{j_1}{z_2^*}^{j_2}{z_3^*}^{j_3}z_4^{j_4})^*.
\end{aligned} 
$$
%We changed $j_1$ to $j_1-1$. 
Hence, we obtain the following recursive equation: 
$$
\begin{aligned}
c_{j_1,j_2,j_3,j_4}(N+1)&=q^4q^{2(j_2+j_3-j_4)}c_{j_1-1,j_2,j_3,j_4)}(N) + q^2q^{2(j_3+j_2)}c_{j_1,j_2-1,j_4,j_5}(N)\\ &\ \ +q^{-2j_4}c_{j_1,j_2,j_3-1,j_4}(N)+c_{j_1,j_2,j_3,j_4-1}(N). 
\end{aligned}
$$
Since $c_{0,0,0,0}(0)=1$, we have 
$$
c_{1,0,0,0}(1)=q^4, \ c_{0,1,0,0}(1)=q^2, \ c_{0,0,1,0}(1)=1, \ c_{0,0,0,1}(1)=1
$$
and 
$$
P_1=A^1{A^1}^*=\begin{pmatrix}
q^4 z_1^*z_1 & q^3z_1^*z_2^* & q^2z_1^*z_3^* & q^2z_1^*z_4 \\
q^3z_2z_1 & q^2z_2z_2^* & qz_2z_3^* & qz_2z_4 \\
q^2z_3z_1 & qz_3z_2^* & z_3z_3^* & z_3z_4 \\
q^2z_4^*z_1 & qz_4^*z_2^* & z_4^*z_3^* & z_4^*z_4
\end{pmatrix}.
$$
For $N=2$ the coefficients are
$$
\begin{aligned}
c_{1,1,0,0}(2)&=q^6(1+q^2), \ c_{1,0,1,0}(2)=q^4(1+q^2), \ c_{1,0,0,1}(2)=q^4(1+q^{-2}) \\
c_{0,1,1,0}(2)&=q^2(1+q^2), \ c_{0,1,0,1}(2)=q^2(1+q^{-2}), \ c_{0,0,1,1}(2)=q^{-2}+1 \\
c_{2,0,0,0}(2)&=q^8, \ c_{0,2,0,0}(2)=q^4, \ c_{0,0,2,0}(2)=1, \ c_{0,0,0,2}(2)=1.  
\end{aligned} 
$$
Let $P_2=A^2{A^2}^*$ which is a square matrix of size 10 where the column vector defining $A^2$ is 
$$
\begin{aligned}
(\psi_{j_1,j_2,j_3,j_4}^2)_{j_1+j_2+j_3+j_4=2}&=(q^3\sqrt{1+q^2}z_2z_1^*, q^2\sqrt{1+q^2}z_3z_1^*, q^2\sqrt{1+q^{-2}}z_4^*z_1^*, q\sqrt{1+q^2}z_3z_2, \\
&\ \ \ \ \ q\sqrt{1+q^{-2}}z_4^*z_2, \sqrt{1+q^{-2}}z_4^*z_3, q^4{z_1^*}^2, q^2{z_2^2}, z_3^2, {z_4^*}^2). 
\end{aligned} 
$$

Consider now the $*$-homomorphism 
$$
\pi: C(\C P_q^{3,\mu})\to C(\C P_q^{3,\mu})/J_2\cong C(SU_q(2))^{\mu}
$$
where $z_1, z_2$ are mapped to $0$ and $z_3,z_4$ to themselves. We also denote by $\pi$ the natural extension of $\pi$ to matrices over $\pol(\C P_q^{3,\mu})$. Note that $\pi(C(S_q^4))=\C\cdot 1$ and $K_0(\pi)([1]_0)=[1]_0, K_0(\pi)([G]_0)=[1]_0$. 

\begin{proposition}
The polynomial projections $K_0(\pi)(P_1)$ and $K_0(\pi)(P_2)$ are generators of $K_0(C(SU_q(2))^{\mu})$. 
\end{proposition}
\begin{proof}
We have 
$$
K_0(\pi)([P_1]_0)= \left[\begin{pmatrix}
z_3z_3^* & z_3z_4 \\ z_4^*z_3^* & z_4^*z_4
\end{pmatrix}\right]_0
$$
$$
K_0(\pi)([P_2]_0)=\left[
\begin{pmatrix}
(1+q^{-2})z_4^*z_3z_3^*z_4 & (1+q^{-2})^{1/2}z_4^*z_3{z_4}^2 & (1+q^{-2})^{1/2}z_4^*z_3{z_3^*}^2 \\
(1+q^{-2})^{1/2}{z_4^*}^2z_3^*z_4 & {z_4^*}^2z_4^2 & {z_4^*}^2{z_3^*}^2 \\
(1+q^{-2})^{1/2}z_3^2z_3^*z_4 & z_3^2z_4^2 & z_3^2{z_3^*}^2
\end{pmatrix}
\right]_0.
$$
Recall that $C(\C P_q^{3,\mu})/J_2$ is  isomorphic to the quantum projective space $C(\C P_q^1)$. In \cite{TYJ91} the 1-summable Fredholm modules of the Podlés sphere are constructed, we apply here the notation from \cite{DAL2010}. The 1-summable Fredholm modules are denoted $\mu_0$ and $\mu_1$, where the corresponding Hilbert spaces are denoted by $H_0$ and $H_1$ respectively. We can determine if the two projections generate the $K$-theory of $C(\C P_q^1)$ by considering the index pairing with these Fredholm modules. 

Using the previous described isomorphism $z_4\mapsto\alpha^*, z_3\mapsto \gamma$ between $C(\C P_q^{3,\mu})/J_2$ and  $C(SU_q(2))^{\mu}$ we obtain that 
$$
[\pi(P_1)]_0=\left[\begin{pmatrix}
\gamma \gamma^* & \gamma\alpha^* \\ \alpha\gamma^* & \alpha\alpha^*
\end{pmatrix}\right]_0, 
$$
$$
[\pi(P_2)]_0 = \left[
\begin{pmatrix}
(1+q^{-2})\alpha\gamma\gamma^*\alpha^* & (1+q^{-2})^{1/2}\alpha\gamma{\alpha^*}^2 & (1+q^{-2})^{1/2}\alpha\gamma{\gamma^*}^2 \\
(1+q^{-2})^{1/2}{\alpha}^2\gamma^*\alpha^* & {\alpha}^2{\alpha^*}^2 & {\alpha}^2{\gamma^*}^2 \\
(1+q^{-2})^{1/2}{\gamma}^2\gamma^*\alpha^* & {\gamma}^2{\alpha^*}^2 & {\gamma}^2{\gamma^*}^2
\end{pmatrix}
\right]_0
$$
in $C(SU_q(2))^{\mu}$. 
Translating this to the notation used in \cite{DAL2010} we have $\alpha\mapsto z_0^*, \gamma\mapsto z_1$. Using the isomorphism  $z_0\mapsto z_0, z_1\mapsto z_1^*$ of $C(\C P_q^1)$ with itself, the above projections become the following in $C(\C P_q^1)$: 
$$
[\pi(P_1)]_0=\left[\begin{pmatrix}
z_1^*z_1 & z_1^*z_0 \\ z_0^*z_1 & z_0^*z_0. 
\end{pmatrix}\right]_0
$$
$$
[\pi(P_2)]_0 = \left[
\begin{pmatrix}
(1+q^{-2})z_0^*z_1^*z_1z_0 & (1+q^{-2})^{1/2}z_0^*z_1^*z_0^2 & (1+q^{-2})^{1/2}z_0^*z_1^*{z_1}^2 \\
(1+q^{-2})^{1/2}{z_0^*}^2 z_1^*z_0 & {z_0^*}^2{z_0}^2 & {z_0^*}^2{z_1}^2 \\
(1+q^{-2})^{1/2}{z_1^*}^2z_1 z_0 & {z_1^*}^2{z_0}^2 & {z_1^*}^2{z_1}^2
\end{pmatrix}
\right]_0
$$
The index paring gives us the following results.  
\begin{table}[H]
\centering
\begin{tabular}{l|cc}
          & \multicolumn{1}{l}{$[\pi(P_1)]_0$} & \multicolumn{1}{l}{$[\pi(P_2)]_0$} \\ \hline
$[\mu_0]$ & 1                         & 1 \\
$[\mu_1]$ & -1                         & -2                                              
\end{tabular}
\end{table}
Below is an example of such a calculation. By geometric series we have 
$$
\begin{aligned} 
\inpro{[\mu_1]}{[\pi(P_2)]_0}&=
Tr_{H_1}(\pi_0-\pi_1)(Tr(\pi(P_2))) \\
&=\sum_{m_1}^{\infty} \left(1-(2+q^{-2}-q^2)q^{2(m_1+1)}(1-q^{2(m_1+1)})-(1-q^{2(m_1+1)})^2-q^{4m_1}\right)\\ 
&=
-\frac{(2+q^{-2}-q^{2})q^2}{1-q^2}+\frac{(2+q^{-2}-q^2)q^4}{1-q^4}-\frac{1}{1-q^4}-\frac{q^4}{1-q^4}+\frac{2q^2}{1-q^2}=-2.
\end{aligned} 
$$ 
Since the matrix $\begin{pmatrix}
1 & 1 \\ -1 & -2
\end{pmatrix}$ is invertible in $M_2(\Z)$ we obtain that $[\pi(P_1)]$ and $[\pi(P_2)]$ generate $K_0(C(\C P_q^1))$. Hence  $K_0(\pi)([P_1])_0$ and $K_0(\pi)([P_2]_0)$ are nontrivial generators of $K_0(C(SU_q(2)^{\mu}))$. 
\end{proof}

To summarize, we obtain a set of elements $[1]_0, [G]_0, [P_i]_0,i=1,2$  in $K_0(C( \C P_q^{3,\mu}))$ which satisfies: 
\begin{enumerate}
\item $[1]_0, [G]_0$ generate $K_0(C(S_q^4))$.
\item $[\pi(P_1)]_0$ and $[\pi(P_2)]_0$ generate $K_0(C(\C P_q^1))$. 
\item $[\pi(1)]_0=[\pi(G)]_0=[1]_0$. 
\end{enumerate}
This indicates a relationship between the K-theory of the total space of the quantum twistor bundle with the K-theory of the base and fibre. 

It would be useful to have a set of $1$-summable Fredholm modules of $C(\CP_q^{3,\mu})$ providing the full pairing between the corresponding $K$-theory and $K$-homology. In this way we would be able to show whether  $[1]_0, [G]_0, [P_i]_0,i=1,2$  generate $K_0(C( \C P_q^{3,\mu}))$. 
Unfortunately, the construction from \cite{DAL2010} is not directly applicable to our case, involving different and seemingly more complicated 
relations between the generators. We hope to address this issue in a separate paper.

\end{document}